\documentclass{lms}

\usepackage{ifpdf}

\ifpdf
\usepackage[colorlinks,hyperindex,linkcolor=blue,urlcolor=black,citecolor=blue,
pdftitle={Traces of analytic uniform algebras on subvarieties and test
  collections}, pdfauthor={Michael A. Dritschel, Daniel Estevez, and
  Dmitry Yakubovich}]{hyperref} \fi

\usepackage{amsmath} \usepackage{amsfonts} \usepackage{amssymb}

\usepackage{color}

\usepackage[english]{babel} \usepackage{enumerate}

\usepackage{graphicx}

\usepackage[abbrev]{amsrefs}

\newcommand{\T}{\mathbb{T}} \newcommand{\D}{\mathbb{D}}
\newcommand{\C}{\mathbb{C}} 
 \newcommand{\N}{\mathbb{N}}
\newcommand{\B}{\mathbb{B}}

\newcommand{\A}{A}

\newcommand{\M}{\mathfrak{M}}

\newcommand{\z}{\mathbf{z}}

\newcommand{\Cont}{\mathcal{C}}

\newcommand{\Cauchy}[2]{\mathcal{C}_{#2}(#1)}

\newcommand{\CauchyMod}[3]{\mathcal{C}_{#3}^{#2}(#1)}

\DeclareMathOperator{\supp}{supp}

\newcommand{\V}{\mathcal{V}}

\newcommand{\SA}{\mathcal{SA}}

\newcommand{\SAa}{\mathcal{SA}_A}

\newtheorem{theorem}{Theorem}[section]
\newtheorem{lemma}[theorem]{Lemma}
\newtheorem{corollary}[theorem]{Corollary}
\newtheorem{proposition}[theorem]{Proposition}
\newnumbered{remark}{Remark} \newnumbered{definition}{Definition}
\newnumbered{claim}{Claim}

\begin{document}

\title[Analytic algebras and test collections] {Traces of analytic
  uniform algebras on subvarieties and test collections}

\author{M. A. Dritschel, D. Est\'evez and D. Yakubovich}

\classno{Primary 46J15; Secondary 30H50, 46J30}

\extraline{The first author's visits to Madrid were partially
  supported by project MTM2011-28149-C02-1, DGI-FEDER, of the Ministry
  of Economy and Competitiveness, Spain (MEC).  The second author was
  supported by the MEC project MTM2015-66157-C2-1-P and the
  Mathematics Department of the Aut\'onoma University of Madrid (UAM).
  His research stay in Newcastle was supported by a grant from UAM.
  The third author was supported by project MTM2015-66157-C2-1-P and
  by ICMAT Severo Ochoa project SEV-2015-0554 of the MEC.  }

\maketitle

\begin{abstract}
  Given a complex domain $\Omega$ and analytic functions
  $\varphi_1,\ldots,\varphi_n : \Omega \to \D$, we find geometric
  conditions allowing us to conclude that $H^\infty(\Omega)$ is
  generated by functions of the form $g \circ \varphi_k$, $g \in
  H^\infty(\D)$.  This is then applied to the construction of an
  extension of bounded holomorphic functions on an analytic
  one-dimensional complex subvariety of the polydisk $\D^n$ to
  functions in the Schur-Agler algebra of $\D^n$, with an estimate on
  the norm of the extension.  Our proofs use an extension of the
  technique of separation of singularities by Havin, Nersessian and
  Ortega-Cerd\'a.
\end{abstract}

\section{Introduction}

\subsection{The statement of main results}

This paper is devoted to the problem of extending a bounded analytic
function from a subvariety of the polydisk $\D^n$ to a bounded
analytic function on the polydisk, as well as a related problem of the
generation of algebras.  Our main motivations come from operator
theory and concern some tests for $K$-spectrality and complete
$K$-spectrality of a Hilbert space linear operator.  These will be
treated in a forthcoming article~\cite{Article2}.

\nocite{*}

Let $\Omega \subset \C$ be a domain and $\Phi : \Omega \to \D^n$ be an
analytic function.  Its image $\V = \Phi(\Omega)$ is an analytic
variety inside $\D^n$ (which may have singular points).  We say that a
complex function $f$ defined on $\V$ is analytic if, for every point
$p \in \V$, there is a neighborhood $U$ of $p$ in $\C^n$ and an
analytic function $F$ on $U$ such that $f|(\V \cap U) = F|(\V \cap
U)$.  We define $H^\infty(\V)$ to be the Banach algebra of bounded
analytic functions on $\V$, equipped with the supremum norm.

A fundamental question is whether it is possible to extend a function
in $H^\infty(\V)$ to a function in $H^\infty(\D^n)$, the Banach
algebra of bounded analytic functions on $\D^n$, also equipped with
the supremum norm.  Since the restriction map $H^\infty(\D^n) \to
H^\infty(\V)$ is a contractive homomorphism, this question asks
whether the image of this homomorphism, $H^\infty(\D^n) | \V$, is all
of $H^\infty(\V)$.

Denote by $\Phi^*$ the pullback by $\Phi$; that is, the map $\Phi^* :
H^\infty(\D^n) \to H^\infty(\Omega)$ defined by $\Phi^*(f) = f \circ
\Phi$.  If this map is onto, i.e, if $\Phi^* H^\infty(\D^n) =
H^\infty(\Omega)$, then every function in $H^\infty(\V)$ can be
extended to a function in $H^\infty(\D^n)$, because if $f \in
H^\infty(\V)$, then $f \circ \Phi \in H^\infty(\Omega)$.  When
$\Phi^*$ is onto, we can find an $F \in H^\infty(\D^n)$ such that $f
\circ \Phi = \Phi^* F = F\circ\Phi$.  This equality implies that $F|\V
= f$, so $F$ extends $f$ to $H^\infty(\D^n)$.

We show that one has $\Phi^* H^\infty(\D^n) = H^\infty(\Omega)$ for a
class of domains $\Omega$ and functions $\Phi$ satisfying certain
geometric conditions, and such that $\Phi$ is injective and $\Phi'$
does not vanish (in this case, $\V$ is an analytic variety).  If the
hypotheses are weaker (in particular, if $\Phi$ is injective and
$\Phi'\ne 0$ only outside a finite subset of $\Omega$), we show that
$\Phi^* H^\infty(\D^n)$ is a finite codimensional subalgebra of
$H^\infty(\Omega)$.  It is easy to see that one cannot get the whole
$H^\infty(\Omega)$ algebra in this case.  As will be seen however,
even under these weaker assumptions, every function in $H^\infty(\V)$
can be extended to a function in $H^\infty(\D^n)$.

We also consider other algebras of functions on $\D^n$ besides
$H^\infty(\D^n)$.  One of these algebras is $\SA(\D^n)$, the Agler
algebra of $\D^n$.  It is the Banach algebra of functions analytic on
$\D^n$ such that the norm
\begin{equation*}
  \|f\|_{\SA(\D^n)} \overset{\text{def}}{=} \sup_{\substack{\|T_j\|
      \leq 1\\\sigma(T_j)\subset \D}} \|f(T_1,\ldots,T_n)\|
\end{equation*}
is finite.  Here the supremum is taken over all tuples
$(T_1,\ldots,T_n)$ of commuting contractions on a Hilbert space such
that the spectra $\sigma(T_j)$ are contained in $\D$
($f(T_1,\ldots,T_n)$ is well defined for such tuples).  Clearly,
$\SA(\D^n)$ is a subset of $H^\infty(\D^n)$ and
$\|f\|_{H^\infty(\D^n)}\leq\|f\|_{\SA(\D^n)}$.  For $n = 1,2$, we have
the equality $\SA(\D^n) = H^\infty(\D^n)$, and the norms coincide.
However, for $n\geq 3$, the norms do not coincide. Also, if $n \geq
3$, it is currently unknown whether or not $\SA(\D^n)$ is a proper
subset of $H^\infty(\D^n)$.  The unit ball of the Agler algebra is
known as the Schur-Agler class. It turns out that it is the proper
analog of the unit ball in $H^\infty(\D)$ (the so called Schur class)
when studying the Pick interpolation problem in $\D^n$.  The
Schur-Agler class also has important applications in operator theory
and function theory.

We can ask whether every function in $H^\infty(\V)$ can be extended to
a function in $\SA(\D^n)$ and whether $\Phi^* \SA(\D^n) =
H^\infty(\Omega)$.  We show that for the class of functions $\Phi$
considered in this article, the answer to the first question is
affirmative, and the answer to the second question is also affirmative
if $\Phi$ is injective and $\Phi'$ does not vanish.

Another interesting algebra is $H^\infty(\mathcal{K}_\Psi)$.  This
algebra, extensively studied in~\cite{DritschelMcCullough}, is
associated with a collection of test functions $\Psi$.  It turns out
that if $\Phi = (\varphi_1,\ldots,\varphi_n): \Omega\to\D^n$ is
injective, then $\{\varphi_1,\ldots,\varphi_n\}$ is a collection of
test functions, which we is also denoted by $\Phi$, and
$H^\infty(\mathcal{K}_\Phi) = \Phi^*\SA(\D^n)$.  Therefore, the
question of whether $H^\infty(\mathcal{K}_\Phi) = H^\infty(\Omega)$,
is a reformulation of the question from the previous paragraph.

If $\Omega$ is a nice domain (say with piecewise smooth boundary), and
$\Phi$ extends by continuity to $\overline{\Omega}$, then we can also
consider the algebra $A(\overline{\Omega})$ of functions analytic in
$\Omega$ and continuous in $\overline{\Omega}$ instead of
$H^\infty(\Omega)$.  The set $\overline{\V} = \Phi(\overline{\Omega})$
is a bordered analytic variety, and we can consider the algebra
$A(\overline{\V})$ of functions analytic in $\V$ and continuous in
$\overline{\V}$.  The extension problem can also be formulated for
these algebras.  One can ask whether every function in
$A(\overline{\V})$ extends to a function in $A(\overline{\D}^n)$, the
algebra of functions analytic in $\D^n$ and continuous in
$\overline{\D}^n$, or to $\SAa(\D^n)\overset{\text{def}}{=} \SA(\D^n)
\cap A(\overline{{\D}^n})$.  Our methods apply to this problem, and so
many of our results have two versions: one for algebras of type
$H^\infty$, another for algebras of functions continuous up to the
boundary.

Another important algebra for us is $\mathcal{H}_\Phi$, the (not
necessarily closed) subalgebra of $H^\infty(\Omega)$ generated by
functions of the form $f\circ \varphi_k$, with $f \in H^\infty(\D)$,
and $k = 1,\ldots,n$:
\begin{equation*}
  \mathcal{H}_\Phi = \bigg\{ \sum_{j=1}^l \prod_{k=1}^n
  f_{j,k}(\varphi_k(z)) : l \in \N, f_{j,k} \in H^\infty(\D) \bigg\}.
\end{equation*}
We have the following algebra inclusions:
\begin{equation}
  \label{eq:inclusions}
  \mathcal{H}_\Phi \subset \Phi^* \SA(\D^n) \subset \Phi^*
  H^\infty(\D^n)
  \subset \Phi^* H^\infty(\V) \subset H^\infty(\Omega).
\end{equation}
The first inclusion follows from the observation that any function on
$\D^n$ of the form $f(z_k)$, with $f \in H^\infty(\D)$, belongs to
$\SA(\D^n)$ by the von~Neumann inequality, as do sums of products of
such functions since $\SA(\D^n)$ is an algebra.  The inclusion $\Phi^*
H^\infty(\D^n) \subset \Phi^* H^\infty(\V)$ holds since if $F \in
H^\infty(\D^n)$, then $F|\V \in H^\infty(\V)$ and $\Phi^* F =
\Phi^*(F|\V)$.

We define $\mathcal{A}_\Phi$ to be the (not necessarily closed)
subalgebra of $A(\overline{\Omega})$ generated by functions of the
form $f \circ \varphi_k$ with $f \in A(\overline{\D})$ and $k =
1,\ldots,n$:
\begin{equation*}
  \mathcal{A}_\Phi = \left\{ \sum_{j=1}^l \prod_{k=1}^n
    f_{j,k}(\varphi_k(z)) : l \in \N, f_{j,k} \in A(\overline{\D})
  \right\}.
\end{equation*}
We likewise have the inclusions
\begin{equation}
  \label{eq:inclusions2}
  \mathcal{A}_\Phi \subset \Phi^* \SAa(\D^n) \subset \Phi^*
  A(\overline{\D}^n)
  \subset \Phi^* A(\overline{\V}) \subset A(\overline{\Omega}).
\end{equation}

Some useful terminology: by an open circular sector in $\C$ with
vertex on a point $z_0$, we mean a set of the form
\begin{equation*}
  \{z \in \C \setminus \{z_0\} : |z - z_0| < r,\ \alpha < \arg(z-z_0) <
  \beta\},
\end{equation*}
where $r > 0$, and $\alpha < \beta < \alpha + 2\pi$.  A closed
analytic arc is understood as the image of the interval $[0,1]$ by an
analytic function, defined and univalent on an open subset of $\C$
containing $[0,1]$.  We denote by $\D_\varepsilon(z_0)$ the open disk
centered at $z_0$ with radius $\varepsilon$.

Let us now introduce the class of functions $\Phi$ to be considered in
this article.

\begin{definition*}
  Let $\Omega$ be a domain whose boundary is a disjoint finite union
  of piecewise analytic Jordan curves such that the interior angles of
  the ``corners'' of $\partial\Omega$ are in $(0,\pi]$.  We say that a
  function $\Phi = (\varphi_1,\ldots,\varphi_n): \overline{\Omega} \to
  \overline{\D}^n$ is \emph{admissible} if $\varphi_k \in
  \A(\overline{\Omega})$, for $k = 1,\ldots,n$, and there is a
  collection of sets $\{J_k\}_{k=1}^n$, where each $J_k$ is a finite
  union of disjoint closed analytic subarcs of $\partial\Omega$, and a
  constant $\alpha$, $0<\alpha\le 1$, such that the following
  conditions are satisfied (see Figure~\ref{fig:0}):
  \begin{enumerate}[(a)]
  \item $\bigcup_{k=1}^n J_k = \partial\Omega$.
  \item $|\varphi_k| = 1$ in $J_k$, for $k = 1,\ldots,n$.
  \item For each $k = 1,\ldots,n$, there exists an open set $\Omega_k
    \supset \Omega$ such that the interior of $J_k$ relative to
    $\partial \Omega$ is contained in $\Omega_k$, $\varphi_k$ is
    defined in $\overline{\Omega}_k$, $\varphi_k \in
    A(\overline{\Omega}_k)$, and $\varphi_k'$ is of class H\"older
    $\alpha$ in $\Omega_k$; i.e.,
    \begin{equation*}
      |\varphi_k'(\zeta) - \varphi_k'(z)| \leq C |\zeta - z|^\alpha,
      \qquad \zeta,z \in \Omega_k.
    \end{equation*}
  \item If $z_0$ is an endpoint of one of the arcs comprising $J_k$,
    then there exists an open circular sector $S_k(z_0)$ with vertex
    on $z_0$ and such that $S_k(z_0) \subset \Omega_k$ and $J_k \cap
    \D_\varepsilon(z_0) \subset S_k(z_0) \cup \{z_0\}$, for some
    $\varepsilon > 0$.  If $z_0$ is a common endpoint of both one of
    the arcs comprising $J_k$ one of the arcs comprising $J_l$, $k
    \neq l$, then we require $(S_k(z_0) \cap S_l(z_0)) \setminus
    \overline{\Omega}$ to be nonempty.
  \item $|\varphi_k'| \geq C > 0$ in $J_k$, for $k = 1,\ldots,n$.
  \item For each $k = 1,\ldots,n$, $\varphi_k|J_k$ is injective and
    $\varphi_k(J_k) \cap \varphi_k(\partial\Omega\setminus J_k) =
    \emptyset$.
  \end{enumerate}
\end{definition*}
\begin{figure}[b]
  \begin{center}
    \includegraphics[width=6.695cm]{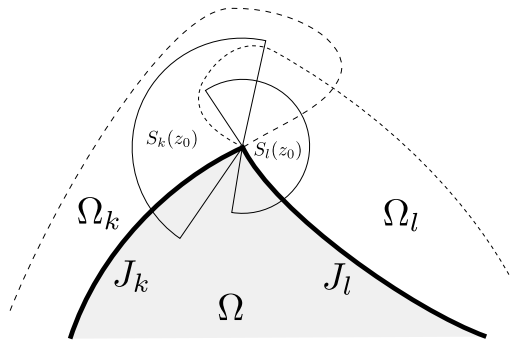}
    \caption{The sets involved in the definition of an admissible
      function}
    \label{fig:0}
  \end{center}
\end{figure}

The hypothesis that $\varphi_k'$ is of class H\"older $\alpha$ in
$\Omega_k$ can be weakened a little by instead requiring that
$\varphi_k'$ be of class H\"older $\alpha$ only in a relative
neighborhood of $J_k$ in $\overline{\Omega}_k$.

It follows from the above hypotheses that if $z_0$ is an endpoint of
one of the arcs comprising $J_k$, then $\varphi_k$ is conformal at
$z_0$.  Since $\varphi_k(\overline\Omega) \subset \overline{\D}$, and
$\varphi_k$ preserves angles, the interior angle of $\partial\Omega$
at $z_0$ must be less than or equal than $\pi$.  This justifies the
assumption on the angles at the corners of $\partial \Omega$.  This is
an important restriction on the class of domains which our methods do
not permit us to relax.

By the Schwarz reflection principle and condition (b), one can always
find sets $\Omega_k$ as in (c) by continuing $\varphi_k$ analytically
across $J_k$.  In general, these sets $\Omega_k$ do not intersect in a
way that permits the construction of the circular sectors required in
(d).  However, if all the interior angles of the corners of $\partial
\Omega$ are greater than $2\pi/3$, then it is easy to see that Schwarz
reflection produces sets $\Omega_k$ which contain such circular
sectors.

Additionally, if $\varphi_k$ is defined only in $\overline{\Omega}$,
$\varphi_k'$ is H\"older $\alpha$ on $\Omega$ and $|\varphi_k'| \geq C
> 0$ in $J_k$, then the extension of $\varphi_k$ to $\Omega_k$ by
Schwarz reflection also satisfies that $\varphi_k'$ is of class
H\"older $\alpha$.

It is easy to check from the definition of an admissible function
$\Phi$ that $\Phi'$ vanishes at most in a finite number of points in
$\overline{\Omega}$ and that there is a finite set $X \subset \Omega$
such that the restriction of $\Phi$ to $\overline{\Omega}\setminus X$
is injective (i.e., $\Phi$ identifies or ``glues'' at most a finite
number of points of $\overline{\Omega}$).

The main results of this article are the following theorems.

\begin{theorem}
  \label{A-Hinfty} If $\Phi : \overline{\Omega} \to \overline{\D}^n$
  is admissible and injective and $\Phi'$ does not vanish on $\Omega$,
  then $\mathcal{H}_\Phi = H^\infty(\Omega)$ and $\mathcal{A}_\Phi =
  A(\overline{\Omega})$.
\end{theorem}

It follows that in this case all the algebras in~\eqref{eq:inclusions}
coincide, as do all those in~\eqref{eq:inclusions2}.

Some of the conditions that we are imposing on $\Phi$ are easily seen
to in fact be necessary for the equality $\Phi^*H^\infty(\D^n) =
H^\infty(\Omega)$, which is weaker than $\Phi^*\SA(\D^n) =
H^\infty(\Omega)$, to hold.  For instance, if $\Phi$ is not injective,
then no function in $\Phi^*H^\infty(\D^n)$ is injective, so this set
cannot be all of $H^\infty(\Omega)$.  Similarly, if $\Phi'(z_0) = 0$
for some $z_0 \in \Omega$, then we have $f'(z_0) = 0$ for every $f \in
\Phi^*H^\infty(\D^n)$, which again implies $\Phi^*H^\infty(\D^n) \neq
H^\infty(\Omega)$.  Finally, if there is a point $z_0 \in \partial
\Omega$ such that $|\varphi_k(z_0)| < 1$ for all functions
$\varphi_k$, then every function in $\Phi^*H^\infty(\D^n)$ is
continuous at $z_0$, so once again $\Phi^*H^\infty(\D^n) \neq
H^\infty(\Omega)$.  It is also easy to show that $\Phi^*
A(\overline{\D}^n) \neq A(\overline\Omega)$ in this case as well.
This motivates conditions (a) and (b) in the definition of an
admissible function.

In the case when $\Phi$ is not injective or $\Phi'$ vanishes at some
points, we prove the following result, which according to the remarks
above is the best that we can hope.

\begin{lemma}
  \label{finite-codim}
  If $\Phi : \overline{\Omega}\to\overline{\D}^n$ is admissible, then
  $\mathcal{H}_\Phi$ is a closed subalgebra of finite codimension in
  $H^\infty(\Omega)$, and $\mathcal{A}_\Phi$ is a closed subalgebra of
  finite codimension in $A(\overline{\Omega})$.
\end{lemma}

In fact, we prove below that $\mathcal{H}_\Phi$ is also weak*-closed
in $H^\infty(\Omega)$ (see Section~\ref{weak-star}).

Regarding the algebras $H^\infty(\V)$ and $A(\overline{\V})$ of
functions defined on the analytic curve $\V$, we prove the following
result.

\begin{theorem}
  \label{analytic-curve}
  If $\Phi : \overline{\Omega}\to\overline{\D}^n$ is admissible, then
  $\Phi^*H^\infty(\V) = \mathcal{H}_\Phi$ and $\Phi^*A(\overline{\V})
  = \mathcal{A}_\Phi$.
\end{theorem}

In this case the first four algebras in~\eqref{eq:inclusions} and the
first four in~\eqref{eq:inclusions2} coincide, while the last
inclusions can be proper, though $\Phi^*H^\infty(\V)$ happens to be
weak*-closed in $H^\infty(\Omega)$, while $\Phi^*A(\overline{\V})$ is
norm closed in $A(\overline{\Omega})$, and both have finite
codimension.

This theorem allows us to prove a result on the extension of functions
in $\V$ to the Agler algebra.

\begin{theorem}
  \label{extension}
  If $\Phi:\overline{\Omega}\to\overline{\D}^n$ is admissible, for
  every $f \in H^\infty(\V)$ there is an $F \in \SA(\D^n)$ such that
  $F|\V = f$ and $\|F\|_{\SA(\D^n)}\leq C\|f\|_{H^\infty(\V)}$, for
  some constant $C$ independent of $f$.  Additionally, if $f \in
  A(\overline{\V})$, then $F$ can be taken to belong to $\SAa(\D^n)$.
\end{theorem}

Our proofs use an extension of the techniques of Havin and Nersessian
and Ortega-Cerd\'a~\cites{HavinNersessian,HavinNersCerda}, which
concern the separation of singularities of bounded analytic functions,
defined on open subsets of $\C$.  Havin and Nersessian prove in
\cite{HavinNersessian} that if $\Omega_1,\Omega_2$ are domains in $\C$
such that their boundaries intersect transversally, then
$H^\infty(\Omega_1\cap\Omega_2) =
H^\infty(\Omega_1)+H^\infty(\Omega_2)$, in the sense that every
function $f \in H^\infty(\Omega_1)$ can be written as $f = f_1 + f_2$
with $f_j \in H^\infty(\Omega_j)$.

The main tool for the proof of the above theorems is the next, which
can be understood as a kind of decomposition result for functions in
$H^\infty(\Omega)$.  Its proof is given in
Section~\ref{proof-compact-operator}.

\begin{theorem}
  \label{compact-operator}
  If $\Phi:\overline{\Omega}\to\overline{\D}^n$ is admissible, then
  there exist bounded linear operators $F_k : H^\infty(\Omega) \to
  H^\infty(\D)$, $k = 1,\ldots,n$, such that the operator defined by
  \begin{equation*}
    f \mapsto f - \sum_{k=1}^n F_k(f) \circ \varphi_k,\qquad f \in
    H^\infty(\Omega),
  \end{equation*}
  is compact in $H^\infty(\Omega)$ and its range is contained in
  $\A(\overline{\Omega})$.  Moreover, $F_k$ maps
  $A(\overline{\Omega})$ into $A(\overline{\D})$, for $k =
  1,\ldots,n$.
\end{theorem}

There is some relationship between our setting and the algebra
generation problem.  Given any finite family
$\Phi=\{\varphi_k\}\subset A(\overline{\Omega})$, one can also
consider the algebras $\overline A_\Phi$, the smallest \textit{norm
  closed} subalgebra of $A(\overline{\Omega})$ containing $\Phi$, and
${\overline H}^\infty_\Phi$, the weak*-closed subalgebra of
$H^\infty(\Omega)$ generated by the family $\Phi$.

\begin{proposition}
  If $\Phi:\overline{\Omega}\to\overline{\D}^n$ is admissible, then
  $\mathcal{A}_\Phi = \overline A_\Phi$ and ${\mathcal H}_\Phi =
  {\overline H}^\infty_\Phi$.
\end{proposition}

\begin{proof}
  It is clear that in general, $\mathcal{A}_\Phi\subset \overline
  A_\Phi$ and ${\mathcal H}_\Phi\subset {\overline H}^\infty_\Phi$.
  By Lemma~\ref{finite-codim}, $\mathcal{A}_\Phi$ is closed in norm,
  and by Lemma~\ref{weak-star-closed}, $\mathcal{H}_\Phi$ is
  weak*-closed.  The equalities now follow.
\end{proof}

Theorem~\ref{A-Hinfty} then implies corresponding results about the
generation of algebras.

\begin{corollary}
  \label{corollary-generation}
  If $\Phi:\overline{\Omega}\to\overline{\D}^n$ is admissible and
  injective and $\Phi'$ does not vanish in $\Omega$, then $\overline
  A_\Phi =A(\overline{\Omega})$ and ${\overline H}^\infty_\Phi =
  H^\infty(\Omega)$.
\end{corollary}

The assertions that $\mathcal{A}_\Phi=A(\overline{\Omega})$ and
${\mathcal H}_\Phi=H^\infty(\Omega)$ are much stronger than just the
fact that $\Phi$ generates algebras $A(\overline{\Omega})$ and
$H^\infty(\Omega)$ (in the weak* sense, in the last case).  For
instance, as was mentioned, the equalities
$\mathcal{A}_\Phi=A(\overline{\Omega})$ and ${\mathcal
  H}_\Phi=H^\infty(\Omega)$ are impossible if there is a point $z_0\in
\partial\Omega$ such that $\max_k|\varphi_k(z_0)|<1$, whereas $\Phi$
still can still generate algebras $A(\overline{\Omega})$ and
$H^\infty(\Omega)$ in this case.  (Notice that for any nonzero
constants $\{\lambda_k\}$, the families $\{\varphi_k\}$ and
$\{\lambda_k\varphi_k\}$ generate the same closed subalgebras of
$A(\overline{\Omega})$ and $H^\infty(\Omega)$.)  In the applications
to operator theory that we consider in~\cite{Article2}, algebra
generation does not suffice, and the assertions that
$\mathcal{A}_\Phi=A(\overline{\Omega})$ and ${\mathcal
  H}_\Phi=H^\infty(\Omega)$ and so Theorem~\ref{compact-operator} play
an important role there.

\subsection{A brief review of previous results on algebra generation
  and continuation}

The study of generators of algebras of the type $A(\overline{\Omega})$
dates back to Wermer~\cite{Wermer}, where he considered pairs of
functions as generators of the algebra $A(K)$ for a compact subset $K$
of a Riemann surface.  Bishop worked on the same problem independently
in~\cite{Bishop}, using a different approach.  In these two articles,
sufficient conditions are given for generation of the whole algebra,
or of a finite codimensional subalgebra.  Several later works are
devoted to giving weaker sufficient conditions; see~\cite{Blumenthal}
and~\cite{SibonyWermer}.  In~\cites{StessinThomas, MathesonStessin},
the $H^p$ closure rather than of the uniform closure of the algebra is
considered, although~\cite{StessinThomas} does also give a result for
the disk algebra.  Even in the simple case of $A(\overline{\D})$,
necessary and sufficient conditions for a pair of functions to
generate the whole algebra are still unknown (see
\cite{ProblemBook}*{Problem 2.32}).

In most articles on algebra generation, it is assumed that the
derivatives of the generators are continuous up to the boundary.  In
our setting, as we remarked after the definition of an admissible
function, it is only necessary that each function $\varphi_k'$ be
H\"older continuous near the arc $J_k$.  In this sense, it seems that
Corollary~\ref{corollary-generation} is a new result.  We also stress
that our results concern the algebra $\mathcal{A}_\Phi$, which is
\emph{a priori} a non-closed algebra smaller than $\overline A_\Phi$,
the smallest closed algebra containing
$\{\varphi_1,\ldots,\varphi_n\}$. In our applications to operator
theory in \cite{Article2}, it is essential that the theorems we have
stated in the Introduction are proved for $\mathcal{A}_\Phi$ rather
than its closure.

There is a case in which Theorems~\ref{A-Hinfty}
and~\ref{compact-operator} are a straightforward consequence of the
results of Havin and Nersessian~\cite{HavinNersessian} on the
separation of singularities of analytic functions.  In this case a
stronger version of Theorem~\ref{compact-operator} can be obtained;
namely, one can prove the existence of bounded linear operators $F_k :
H^\infty(\Omega) \to H^\infty(\D)$ such that
\begin{equation}
  \label{eq:f-sum}
  f = \sum_{k=1}^n F_k(f) \circ \varphi_k,\qquad f \in H^\infty(\D).
\end{equation}
This case is as follows.  Assume that there are simply connected
domains $D_k$, $k = 1,\ldots,n$, such that $\Omega = \bigcap_k D_k$,
and that $\varphi_k$ are conformal maps from $D_k$ onto $\D$.  If the
boundaries of the domains $D_k$ intersect transversally, by
\cite{HavinNersessian}*{Example 4.1} there are bounded linear
operators $G_k : H^\infty(\Omega) \to H^\infty(D_k)$ such that $f =
\sum_k G_k(f)$ for every $f \in H^\infty(\Omega)$.  If we put $F_k(f)
= G_k(f) \circ \varphi_k^{-1}$, then we get~\eqref{eq:f-sum}.  From
this, the equality $\mathcal{H}_\Phi = H^\infty(\D)$ follows
trivially.  It is not difficult to see that such operators $F_k$ map
$A(\overline{\Omega})$ into $A(\overline{\D})$, so we also get the
equality $\mathcal{A}_\Phi = A(\overline{\Omega})$.

The case when $\Omega = \D$ is important.  Then $\V = \Phi(\D)$ is
called an analytic disk inside the polydisk.  It is a particular kind
of hyperbolic analytic curve, the theory of which has been treated
extensively in the literature.  The function theory of this curves and
its relation with finite codimensional subalgebras of holomorphic
functions was studied by Agler and McCarthy in
\cite{AglerMcCarthyHyp}.  A classification of the finite codimensional
subalgebras of a function algebra which is related to the one that we
use in the proof of Theorem~\ref{A-Hinfty} was given by Gamelin
in~\cite{Gamelin}.

The problem of extension of a bounded analytic function defined on an
analytic curve $\V\subset \D^n$ to the polydisk $\D^n$ dates back to
Rudin and to Stout (see~\cite{Stout}), and was also treated by
Polyakov in~\cite{Polyakov}, and more generally by Polyakov and
Khenkin in~\cite{PolyakovKhenkin}.  The book~\cite{HenkinLeiterer} by
Henkin and Leiterer also treats the extension of bounded analytic
functions defined on subvarieties in a fairly general context.  In
these works, the subvariety $\V$ is assumed to be extendable to a
neighborhood of $\overline{\D}^n$, which means that there is a larger
analytic subvariety $\widetilde{\V}$ of a neighborhood of
$\overline{\D}^n$ such that $\V = \widetilde{\V}\cap \D^n$.  This is
in contrast to a function $\Phi$ meeting our requirements (a)--(f)
``in general position'', in which case the variety $\V=\Phi(\Omega)$
does not extend to a larger analytic variety $\widetilde{\V}$.

The works of Amar and Charpentier~\cite{AmarCharpentier} and of Chee
\cites{Chee76, Chee83} do deal with the setting when this extension of
$\V$ may be absent.  In~\cite{AmarCharpentier}, extensions by bounded
analytic functions to bidisks are considered, whereas the papers
\cites{Chee76, Chee83} concern the case of an analytic variety $\V$ of
codimension~$1$ in a polydisk $\D^n$ and therefore can be compared
with our results only for $n=2$.  Theorem 1.1 in~\cite{Chee83} implies
that in our setting, for the case of $n=2$, every $f \in H^\infty(\V)$
can be extended to an $F\in H^\infty(\D^2)$ such that $F|\V = f$.

For the case of $\V=\Phi(\D)$, where $\Phi:\D\to \D^n$ extends to a
neighborhood of $\overline{\D}$, a necessary and sufficient condition
for the property of analytic bounded extension was given by Stout in
\cite{Stout}, in which case, at least one $\varphi_k$ must be a finite
Blaschke product.

If $\widetilde{\V}$ is an analytic curve in a neighborhood of
$\overline{\D}^n$ such that there is a biholomorphic map
$\widetilde{\Phi}$ of a domain $G\subset \C$ onto $\widetilde{\V}$ and
$\V=\widetilde{\V}\cap \D^n$, the set
$\Omega=\widetilde{\Phi}^{-1}(\V)$ is connected and
$\Phi=\widetilde{\Phi}|\Omega$, then typically all the above
conditions (a)--(e) on $\Phi$ are satisfied, whereas (f) is an
additional requirement.  In this case, if $\Omega$ is simply connected
and $\widehat\Phi=\Phi\circ \eta :\D\to \V$, where $\eta$ is a Riemann
mapping of $\D$ onto $\Omega$, then $\widehat\Phi$ does not continue
analytically to a larger disk unless $\Omega$ has analytic boundary.
In other words, there are cases when $\V$ has an extension to a larger
subvariety whereas $\Phi$ does not extend.

Bounded extensions to an analytic polyhedron $W$ in $\C^n$ from a
subvariety $\V$ of arbitrary codimension were studied by Adachi,
Andersson and Cho in~\cite{AdachiEtl}.  It was assumed there that $\V$
is continuable to a neighborhood of $W$.  Notice that polydisks are
particular cases of analytic polyhedra.

The property of the bounded extension of $H^\infty$ functions does not
hold in general, and one can find several counterexamples in the
literature, see~\cites{Alexander69, DiedMazz97, DiedMazz01 ,Mazzilli}.

There are also many papers in the literature that deal with bounded
extensions in the context of strictly pseudoconvex domains or domains
with smooth boundary (the polydisk does not belong to these classes).
See Diederich and Mazzilli~\cites{DiedMazz97, DiedMazz01} and the
recent paper by Alexandre and Mazzilli~\cite{AlexandreMazzilli}.
Holomorphic extensions have also been studied extensively in different
contexts of $L^p$ norms; we refer to Chee~\cites{Chee83, Chee87} for
the case of the polydisk.  See also the review~\cite{Adachi} by
Adachi, the paper~\cite{AlexandreMazzilli} and references therein for
a more complete information.

The above-mentioned papers use diverse techniques from several complex
variables, such as the Cousin problem and integral representations for
holomorphic functions.  In another group of papers, interesting
results around the problem of bounded continuation are obtained using
the tools of operator theory and the theory of linear systems.  Agler
and McCarthy~\cite{AglerMcCarthyNorm} treat the bounded extension
property with preservation of norms for the bidisk $\D^2$.  See also
\cite{GuoHuangWang2008} for partial results for the case of tridisks
and general polydisks.  It seems that very few varieties $\V$ have
this norm preserving extension property.

In~\cite{Knese2010}, Knese studies the existence of bounded extensions
from distinguished subvarieties of $\D^2$ (without preservation of
norms).  His approach is based on certain representations of
two-variable transfer functions and permits him to give concrete
estimates of the constants.

The same problem can also be studied for the ball $\B^n$ instead of
the polydisk.  In~\cite{AlpayPutinarVinnikov}, Alpay, Putinar and
Vinnikov use reproducing kernel Hilbert space techniques to show that
a bounded analytic function defined on an analytic disk in $\B^n$ can
be extended to $\B^n$.  Indeed they show that it can be extended to an
element of the multiplier algebra of the Drury-Arveson space
$H^2(\B^n)$; this algebra is properly contained in $H^\infty(\B^n)$.
See also~\cite{DavidsonHartzSh2015} for further examples and
counterexamples, and the relationship with the complete
Nevanlinna-Pick property.

The extension problem is also treated in~\cite{KerrMcCarthyShalit},
where it appears as a consequence of isomorphism of certain multiplier
algebras of analytic varieties.  Some problems considered there
resemble those we consider on the pullback by $\Phi$.

Our approach differs from the approaches described above, in that it
relies on techniques inspired by the Havin-Nersessian work, certain
compactness arguments, which show that some subalgebras of $H^\infty$
have finite codimension, and the study of maximal ideals and
derivations in $H^\infty$.  An important aspect distinguishing it from
earlier results, is that we can prove continuation to $\SA(\D^n)$.  If
$\SA(\D^n)$ is strictly contained in $H^\infty(\D^n)$ (it is not known
whether this is true), then our results are stronger.  Indeed, we
prove even more: it follows from the proof of Theorem~\ref{extension}
that there is closed subspace of finite codimension in $H^\infty(\V)$
such that every function in this space can be extended to a function
of the form $F_1(z_1) + \cdots + F_n(z_n)$, where $F_j \in
H^\infty(\D)$ for $j = 1,\ldots,n$.

Following on from the work in~\cite{HavinNersessian}, the papers
\cite{HavinNersCerda} by Havin, Nersessian and Ortega-Cerd\`a, and
\cite{Khavin} by Havin prove separation of singularities under weaker
hypothesis. It would be interesting to know if some of these results,
in particular the examples at the end of~\cite{Khavin} can be used to
extend what we do.

\subsection{The organization of the paper}

Section~\ref{proof-A-Hinfty} is devoted to the proof of
Theorem~\ref{A-Hinfty}.  The proof of this theorem uses
Theorem~\ref{compact-operator}, which is proved in
Section~\ref{proof-compact-operator}.  To prove
Theorem~\ref{compact-operator}, we need to use some facts about weakly
singular integral operators, which are given in
Section~\ref{weakly-singular-integral}.  In Section~\ref{weak-star},
we deal with the weak*-closedness of the algebras that we are
treating.  In Section~\ref{glued-subalgebras} we consider finite
codimensional subalgebras of a particular kind which we call ``glued
subalgebras''.  These will be a key tool in the sequel.
Section~\ref{more-proofs} contains the proofs of
Theorems~\ref{analytic-curve} and~\ref{extension}.  Finally, in
Section~\ref{continuous-families} we give some lemmas regarding
families of functions $\Phi_\varepsilon$ that depend continuously on a
parameter $\varepsilon$.  These results are not used elsewhere in this
article, but they will be essential in~\cite{Article2}.  We place them
here, because their proof uses the details of the proof of
Theorem~\ref{compact-operator}.

\section{The proof of Theorem~\ref{A-Hinfty} (modulo
  Theorem~\ref{compact-operator})}
\label{proof-A-Hinfty}

In this section we prove Theorem~\ref{A-Hinfty}.  This requires
Theorem~\ref{compact-operator}, which was stated above and is proved
in Section~\ref{proof-compact-operator}.  As a first step, we obtain
Lemma~\ref{finite-codim}, stated in the Introduction, which is a
simple consequence of Theorem~\ref{compact-operator}.

\begin{proof}[of Lemma~\ref{finite-codim}]
  By Theorem~\ref{compact-operator} and the standard theory of
  Fredholm operators, the range of the operator $f \mapsto \sum F_k(f)
  \circ \psi_k$, $f \in H^\infty(\Omega)$, is a closed subspace of
  finite codimension in $H^\infty(\Omega)$.  Since $\mathcal{H}_\Phi$
  contains this range, we get that $\mathcal{H}_\Phi$ is a closed
  subalgebra of finite codimension in $H^\infty(\Omega)$.  To obtain
  the analogous result for $\mathcal{A}_\Phi$, we just consider the
  restriction of the operator $f \mapsto \sum F_k(f) \circ \psi_k$ to
  $A(\overline{\Omega})$.
\end{proof}

The proof of Theorem~\ref{A-Hinfty} follows essentially from
Lemma~\ref{finite-codim}, together with the application of some Banach
algebra techniques.  We now recall some basic facts about the maximal
ideal space of $H^\infty(\Omega)$ which are used in the proof.  We
refer to~\cite{Hoffman}*{Chapter 10} for a detailed discussion of the
Banach algebra structure of $H^\infty(\D)$.  The properties of
$H^\infty(\Omega)$ for a finitely connected domain $\Omega$ are
similar.  We denote by $\M(H^\infty(\Omega))$ the space of all complex
homomorphisms on $H^\infty(\Omega)$, endowed with the weak* topology
inherited as a subspace of the dual space $(H^\infty(\Omega))^*$.  It
is a compact Hausdorff space.

We denote by $\z$ the identity function on $\Omega$, i.e., $\z(z) =
z$.  For any complex homomorphism $\psi \in \M(H^\infty(\Omega))$,
either $\psi(\z) \in \Omega$ or $\psi(\z) \in
\partial\Omega$.  If $\psi(\z) = z_0 \in \Omega$, then $\psi(f) =
f(z_0)$ for every $f \in H^\infty(\Omega)$.  If $\psi(\z) = z_0 \in
\partial\Omega$, then we can assert that $\psi(f) = f(z_0)$
for every $f \in H^\infty(\Omega)$ that extends by continuity to $z_0$
(the proof for $\Omega = \D$, given in~\cite{Hoffman}, easily adapts
to any finitely connected domain).

A linear functional $\eta \in (H^\infty(\Omega))^*$, it is called a
\emph{derivation} at $\psi \in \M(H^\infty(\Omega))$ if
\begin{equation*}
  \eta(fg) = \eta(f)\psi(g) + \psi(f)\eta(g),\qquad \forall f,g\in
  H^\infty(\Omega).
\end{equation*}
It is easy to see that if $\eta$ is a derivation at $\psi$ with
$\psi(\z) = z_0 \in \Omega$, then $\eta(f) = \eta(\z)f'(z_0)$ for
every $f \in H^\infty(\Omega)$ (one must check first that $\eta(1) =
0$ and then write $f = f(z_0) + f'(z_0)(\z-z_0) + (\z-z_0)^2g$ with $g
\in H^\infty(\Omega)$).  Derivations at $\psi$ with
$\psi(\z)\in\partial\Omega$ have the following somewhat similar
property.

\begin{lemma}
  \label{derivation-T}
  Let $f \in H^\infty(\Omega)$ be continuous at $z_0 \in
  \partial\Omega$ with $(f - f(z_0))/(\z - z_0) \in
  H^\infty(\Omega)$.  If $\eta$ is a derivation in $H^\infty(\Omega)$
  at $\psi \in \M(H^\infty(\Omega))$ with $\psi(\z) = z_0$, then
  \begin{equation*}
    \eta(f) = \eta(\z) \psi\Big(\frac{f-f(z_0)}{\z - z_0}\Big).
  \end{equation*}
\end{lemma}

\begin{proof}
  We just compute
  \begin{equation*}
    \begin{split}
      \eta(f) &= \eta(f - f(z_0)) =
      \eta\Big((\z-z_0)\frac{f-f(z_0)}{\z-z_0}\Big) \\ &= \psi(\z-z_0)
      \eta\Big(\frac{f-f(z_0)}{\z-z_0}\Big) + \eta(\z-z_0)
      \psi\Big(\frac{f-f(z_0)}{\z-z_0}\Big) = \eta(\z)
      \psi\Big(\frac{f-f(z_0)}{\z-z_0}\Big).
    \end{split}
  \end{equation*}
\end{proof}

\begin{lemma}
  \label{lemma-glued}
  If $\Phi:\overline{\Omega}\to\overline{\D}^n$ is admissible and
  $\psi_1\neq\psi_2$ are in $\M(H^\infty(\Omega))$ and satisfy
  $\psi_1(f) = \psi_2(f)$ for every $f\in \mathcal{H}_\Phi$, then
  $\psi_1(\z),\psi_2(\z)\in\Omega$, $\psi_1(\z) \neq \psi_2(\z)$ and
  $\Phi(\psi_1(\z)) = \Phi(\psi_2(\z))$.  The same is true if
  $H^\infty(\Omega)$ is replaced by $A(\overline{\Omega})$ and
  $\mathcal{H}_\Phi$ is replaced by $\mathcal{A}_\Phi$.
\end{lemma}

\begin{proof}
  Since by assumption $\varphi_k \in A(\overline{\Omega})$, we have
  $\psi_j(\varphi_k) = \varphi_k(\psi_j(\z))$ for $j=1,2$, $k =
  1,\ldots,n$.  Therefore, $\Phi(\psi_1(\z)) = \Phi(\psi_2(\z))$,
  because the functions $\varphi_k$ belong to $\mathcal{H}_\Phi$.  Let
  $z_j = \psi_j(\z) \in \overline{\Omega}$.  If $z_1\in
  \partial\Omega$, then by condition (f), $z_2=z_1$,
  and hence $\psi_1(f) = \psi_2(f)$ for all $f\in
  A(\overline{\Omega})$.  Take an $f \in H^\infty(\Omega)$ and put $g
  = \sum_{k=1}^N F_k(f)\circ\varphi_k$, where $F_k$ are as in
  Theorem~\ref{compact-operator}. Then $f-g\in A(\overline{\Omega})$
  and $g \in \mathcal{H}_\Phi$.  Therefore, we have
  $\psi_1(f-g)=\psi_2(f-g)$, and also $\psi_1(g)=\psi_2(g)$.  It
  follows that $\psi_1(f) = \psi_2(f)$, so that $\psi_1 = \psi_2$,
  because $f$ was arbitrary. This contradicts our assumption.  Hence
  $\psi_j(\z) \in \Omega$ for $j=1,2$ and, since $\psi_1\neq \psi_2$,
  it must happen that $\psi_1(\z) \neq \psi_2(\z)$. The reasoning for
  $A(\overline{\Omega})$ is the same.
\end{proof}

\begin{lemma}
  \label{lemma-zero}
  If $\Phi:\overline{\Omega}\to\overline{\D}^n$ is admissible and
  $\eta \neq 0$ is a derivation of $H^\infty(\Omega)$ at
  $\psi\in\M(H^\infty(\Omega))$ such that $\eta(f) = 0$ for every
  $f\in\mathcal{H}_\Phi$, then $\psi(\z) \in \Omega$ and
  $\Phi'(\psi(\z)) = 0$.  The same is true if $H^\infty(\Omega)$ is
  replaced by $A(\overline{\Omega})$ and $\mathcal{H}_\Phi$ is
  replaced by $\mathcal{A}_\Phi$.
\end{lemma}

\begin{proof}
  We consider two cases according to whether $\psi(\z)$ belongs to
  $\Omega$ or to $\partial\Omega$.  The case when $\psi(\z) \in
  \Omega$ is clear: since $\varphi_k \in \mathcal{H}_\Phi$, we have $0
  = \eta(\varphi_k) = \varphi_k'(\psi(\z))$, and so $\Phi'(\psi(\z)) =
  0$.

  Now we show that the case $\psi(\z) \in \partial\Omega$ cannot
  happen.  Here we also distinguish two cases according to whether
  $\eta(\z)$ is zero or not.  If $\eta(\z) \neq 0$, then we take $k
  \in \{1,\ldots,n\}$ such that $\psi(\z) \in J_k$.  Since $\varphi_k$
  is derivable at $\psi(\z)$, we have $\eta(\varphi_k) =
  \eta(\z)\varphi_k'(\psi(\z))$ by Lemma~\ref{derivation-T}.
  Therefore, $\varphi_k'(\psi(\z)) = 0$, because $\varphi_k \in
  \mathcal{H}_\Phi$.  This contradicts condition (e) in the definition
  of an admissible family.

  In the case when $\eta(\z) = 0$, we get $\eta(f) = 0$ for every $f$
  analytic on some neighborhood of $\overline{\Omega}$.  This implies
  $\eta(f) = 0$ for every $f \in A(\overline{\Omega})$, because
  functions analytic on $\overline{\Omega}$ are dense in
  $A(\overline{\Omega})$.  Now take $f \in H^\infty(\Omega)$ and put
  $g = \sum_{k=1}^nF_k(f) \circ \varphi_k$, where $F_k$ are as in
  Theorem~\ref{compact-operator}.  Then $f-g \in A(\overline{\Omega})$
  and $g \in \mathcal{H}_\Phi$.  This implies that $0 = \eta(g) =
  \eta(g) + \eta(f-g) = \eta(f)$.  Therefore, $\eta = 0$, a
  contradiction.

  The proof for $A(\overline{\Omega})$ follows similar steps, and is
  indeed even easier.
\end{proof}

We are now ready to give the proof of Theorem~\ref{A-Hinfty}.

\begin{proof}[of Theorem~\ref{A-Hinfty}]
  We first show that $\mathcal{H}_\Phi = H^\infty(\Omega)$.  By
  Lemma~\ref{finite-codim}, $\mathcal{H}_\Phi$ is a closed unital
  subalgebra of $H^\infty(\Omega)$ of finite codimension.  Let us
  assume by way of contradiction that $\mathcal{H}_\Phi \neq
  H^\infty(\Omega)$.  Gorin proves in~\cite{Gorin} that every proper
  subalgebra of a commutative complex algebra is contained in a
  subalgebra of codimension one.  Therefore, there exists a closed
  unital subalgebra $A$ of $H^\infty(\Omega)$ of codimension one such
  that $\mathcal{H}_\Phi \subset A$.

  We use the classification of the closed unital subalgebras of
  codimension one in a Banach algebra, which also is given
  in~\cite{Gorin}.  According to this classification, $A$ must have
  one of the following two possible forms:

  \begin{enumerate}[(a)]
  \item $A = \{f \in H^\infty : \psi_1(f) = \psi_2(f)\}$, for some
    $\psi_1,\psi_2 \in \M(H^\infty(\Omega))$, $\psi_1 \neq \psi_2$.

  \item $A = \ker \eta$, where $\eta \neq 0$ is a derivation at some
    $\psi \in \M(H^\infty(\Omega))$.
  \end{enumerate}

  We show that each of these two cases leads to a contradiction.  In
  the case (a), Lemma~\ref{lemma-glued} shows that $\psi_1(\z)\neq
  \psi_2(\z) \in \Omega$, yet $\Phi(\psi_1(\z)) = \Phi(\psi_2(\z))$.
  Since $\Phi$ is injective, we get a contradiction.

  In the case (b), Lemma~\ref{lemma-zero} shows that $\psi(\z) \in
  \Omega$ and $\Phi'(\psi(\z)) = 0$.  This is a contradiction, because
  $\Phi'$ does not vanish in $\Omega$.

  The proof of the equality $\mathcal{A}_\Phi = A(\overline{\Omega})$
  is identical.
\end{proof}

A few comments about the proof of Theorem~\ref{A-Hinfty} are in order.
The first is about the classification of the derivations of
$H^\infty(\Omega)$.  We treat the case $\Omega = \D$, since the case
of a finitely connected domain $\Omega$ is similar.  We have already
described the derivations at points $\psi \in \M(H^\infty(\D))$ such
that $\psi(\z) \in \D$ and given some properties about those
derivations such that $\psi(\z) \in \T$.  The first question is
whether there exists any such (non-zero) derivations ``supported on
$\T$'', and whether the case of a derivation $\eta$ such that
$\eta(\z)=0$ but $\eta\neq 0$ that appeared in the proof of the
theorem can really happen.

It is important to remark the existence of analytic disks inside each
of the fibers of $\M(H^\infty(\D))$ that project into $\T$ under the
map $\psi \mapsto \psi(\z)$ (once again, we refer the reader to
\cite{Hoffman}).  Thus there are maps of the form $\Psi_{\cdot} : \D
\to \M(H^\infty(\D))$ such that for every $\lambda \in \D$ the point
$\Psi_\lambda \in \M(H^\infty(\D))$ lies in the same fiber (i.e.,
$\Psi_\lambda(\z)$ is constant in $\lambda$) and such that the map
$f(\lambda) \mapsto \Psi_\lambda(f)$ is an algebra homomorphism of
$H^\infty(\D)$ onto $H^\infty(\D)$.  This map $\Psi$ endows its image
$\mathcal D$ in $\M(H^\infty(\D))$ with an analytic structure.  The
complex derivative according to the analytic structure of $\mathcal D$
gives a (non-zero) derivation at each of the points in $\mathcal D$
(explicitly, these are maps $f \mapsto (d/d\lambda)
|_{\lambda=\lambda_0} \Psi_\lambda(f)$).  Clearly, if $\eta$ is one of
these derivations, then $\eta(\z)=0$, because $\z$ is constant on each
of the fibers over $\T$, and hence in $\mathcal D$.  However, we do
not know whether these derivations are (up to a constant multiple) the
only ones that exist over points of $\mathcal D$, or whether there
exist (non-zero) derivations on points which do not belong to such
analytic disks.  It seems that there is not much information about the
classification of the derivations of $H^\infty(\D)$ in the literature.

Another comment is that one could use the results of
Section~\ref{weak-star} to simplify somewhat the proof of the
Theorem~\ref{A-Hinfty}.  If we know that the algebra $A$ is
weak*-closed, then we only need to consider weak*-continuous complex
homomorphisms and derivations.

\section{Some lemmas about weakly singular integral operators}
\label{weakly-singular-integral}

\begin{definition*}
  We say that a domain $\Omega \subset \C$ satisfies the inner
  chord-arc condition if there is a constant $C > 0$ depending only on
  $\Omega$ such that for every $\zeta,z \in \overline{\Omega}$ there
  is a piecewise smooth curve $\gamma(\zeta,z)$ which joins $\zeta$
  and $z$, is contained in $\Omega$ except for its endpoints, and
  whose length is smaller or equal than $C|\zeta-z|$.
\end{definition*}

\begin{lemma}
  \label{weak-singularity}
  Let $U \subset \C$ be a domain satisfying the inner chord-arc
  condition, $\varphi \in \A(\overline{U})$ with $\varphi'$ of class
  H\"older $\alpha$, $0 < \alpha \leq 1$ in $U$ $($so that $\varphi'$
  extends to $\overline{U}$ by continuity$)$.  Let $K \subset
  \overline{U}$ be compact and $\Omega \subset U$ be a domain. Assume
  that $\varphi(\zeta) \neq \varphi(z)$ if $\zeta \in K$ and $z \in
  \overline{\Omega}\setminus\{\zeta\}$, and that $\varphi'$ does not
  vanish in $K$.

  Then, the function
  \begin{equation*}
    G(\zeta,z) = \frac{\varphi'(\zeta)}{\varphi(\zeta)-\varphi(z)} -
    \frac{1}{\zeta - z}
  \end{equation*}
  satisfies
  \begin{equation*}
    |G(\zeta,z)| \leq C|\zeta - z|^{\alpha-1}, \qquad \zeta \in K,\ z
    \in \overline{\Omega}\setminus\{\zeta\}.
  \end{equation*}
\end{lemma}

\begin{proof}
  Let us first check that
  \begin{equation}
    \label{eq:lemma4-*}
    \left|\frac{\varphi(\zeta) - \varphi(z)}{\zeta - z}\right| \geq
    C_1 > 0,
    \qquad
    \zeta \in K,\ z \in \overline{\Omega}\setminus\{\zeta\}.
  \end{equation}
  To see this, put
  \begin{equation*}
    h(\zeta,z) =
    \begin{cases}
      \frac{\varphi(\zeta) - \varphi(z)}{\zeta - z}, & \text{if }
      \zeta \in K, z \in
      \overline{\Omega}\setminus\{\zeta\},\\
      \varphi'(\zeta), & \text{if } \zeta = z \in K.
    \end{cases}
  \end{equation*}
  Since $h$ is continuous on the compact set $K \times
  \overline{\Omega}$ and does not vanish, we get $|h| \geq C_1 > 0$,
  which implies~\eqref{eq:lemma4-*}.

  If $\zeta \in K$ and $z \in \overline{\Omega}\setminus\{\zeta\}$,
  let $\gamma(\zeta,z) \subset U$ be an arc joining $\zeta$ and $z$
  and whose length is comparable to $|\zeta-z|$. Then
  \begin{equation}
    \label{eq:lemma4-**}
    \begin{split}
      |\varphi(z) - \varphi(\zeta) - \varphi'(\zeta)(z-\zeta)| &=
      \bigg|\int_{\gamma(\zeta,z)} \big(\varphi'(u) -
      \varphi'(\zeta)\big)\,du\bigg| \leq
      C_2\int_{\gamma(\zeta,z)}|u-\zeta|^\alpha|du| \\
      &\leq C_3|z-\zeta|^{\alpha+1}.
    \end{split}
  \end{equation}
  Using~\eqref{eq:lemma4-*} and~\eqref{eq:lemma4-**}, we get with $C =
  C_3/C_1$,
  \begin{equation*}
    |G(\zeta,z)| = \frac{|\varphi(z) - \varphi(\zeta) -
      \varphi'(\zeta)(z-\zeta)|}{|\varphi(\zeta)-\varphi(z)||\zeta-z|}
    \leq C|\zeta-z|^{\alpha-1},
  \end{equation*}
  which proves the lemma.
\end{proof}

The following lemma on the compactness of weakly singular integral
operators may be well know to specialists.  It appears throughout the
literature in different forms.  The one given here is similar to that
in~\cite{Kress}*{Theorem~2.22}, and it can be proved in the same way.
Hence, we omit the proof.

\begin{lemma}
  \label{weak-singular-integral}
  Let $\Omega \subset \C$ be bounded domain, $K \subset \C$ a compact
  piecewise smooth curve, and $G(\zeta,z)$ continuous in $(K \times
  \overline{\Omega})\setminus\{(\zeta,\zeta): \zeta \in K\}$ with
  $|G(\zeta,z)| \leq C|\zeta-z|^{-\beta}$ for some $\beta < 1$ and
  every $\zeta \in K$, $z\in \overline{\Omega}\setminus\{\zeta\}$.

  Then the operator
  \begin{equation}
    \label{eq:integral-operator}
    (T\psi)(z)=\int_K G(\zeta,z)\psi(\zeta)d\zeta
  \end{equation}
  defines a compact operator $T: L^\infty(K) \to
  \Cont(\overline{\Omega})$.
\end{lemma}

If $\Gamma \subset \C$ is a piecewise smooth closed Jordan arc $\psi
\in L^\infty(\Gamma)$, and $\varphi$ and $\varphi'$ are defined and
continuous in $\Gamma$, we define the modified Cauchy integral
\begin{equation*}
  \CauchyMod\psi\varphi \Gamma (z) = \int_\Gamma
  \frac{\varphi'(\zeta)}{\varphi(\zeta)-z} \psi(\zeta)\,d\zeta.
\end{equation*}
The function $\CauchyMod\psi \varphi \Gamma$ is analytic in $\C
\setminus \overline{\varphi(\Gamma)}$.  We write $\Cauchy \psi \Gamma$
for the usual Cauchy transform (i.e., when $\varphi(z) = z$).

\section{Proof of Theorem~\ref{compact-operator}}
\label{proof-compact-operator}

The following two lemmas are used in the proof of
Theorem~\ref{compact-operator}.

\begin{lemma}
  \label{diff-compact}
  Under the hypotheses of Theorem~\ref{compact-operator}, if $\Gamma$
  is a piecewise smooth closed arc contained in $\overline{\Omega}_k$,
  then the operator defined by
  \begin{equation}
    \label{eq:psimapsto}
    \psi \mapsto \CauchyMod{\psi}{\varphi_k}{\Gamma} \circ \varphi_k -
    \Cauchy{\psi}{\Gamma}
  \end{equation}
  maps $L^\infty(\Gamma)$ into $A(\overline{\Omega})$ and is compact.
\end{lemma}

\begin{proof}
  We compute
  \begin{equation}
    \label{eq:lemma-simple-*}
    \CauchyMod{\psi}{\varphi_k}{\Gamma} \circ \varphi_k -
    \Cauchy{\psi}{\Gamma}
    =
    \int_{\Gamma} \left[\frac{\varphi_k'(\zeta)}{\varphi_k(\zeta) -
        \varphi_k(z)}
      - \frac{1}{\zeta-z}\right]\psi(\zeta)\,d\zeta.
  \end{equation}

  Using Lemma~\ref{weak-singularity} with $U=\Omega_k$, we have
  \begin{equation}
    \label{eq:wsing}
    \left|\frac{\varphi_k'(\zeta)}{\varphi_k(\zeta) - \varphi_k(z)}
      - \frac{1}{\zeta-z}\right| \leq C|\zeta-z|^{\alpha-1},
    \qquad
    \zeta \in \Gamma,\ z \in \overline{\Omega}\setminus\{\zeta\}.
  \end{equation}
  By Lemma~\ref{weak-singular-integral}, we see that the operator
  defined by~\eqref{eq:psimapsto} is compact from $L^\infty(\Gamma)$
  to $\Cont(\overline{\Omega})$.  Since its image clearly consists of
  analytic functions, the Lemma follows.
\end{proof}

\begin{lemma}
  \label{lemma-simple}
  Under the hypotheses of Theorem~\ref{compact-operator}, let
  $\widehat{J}_k$ be a closed arc contained the interior of $J_k$
  relative to $\partial \Omega$.  If $\psi \in
  L^\infty(\widehat{J}_k)$ and $\Cauchy \psi {\widehat{J}_k} \in
  H^\infty(\C\setminus\widehat{J}_k)$, then the modified Cauchy
  integral $\CauchyMod \psi {\varphi_k} {\widehat{J}_k}$ belongs to
  $H^\infty(\C\setminus\varphi_k(\widehat{J}_k))$.
\end{lemma}

\begin{proof}
  We must verify that $\CauchyMod \psi {\varphi_k} {\widehat{J}_k}$ is
  bounded in $\C \setminus \varphi_k(\widehat{J}_k)$.  It is enough to
  check that it is bounded in $\varphi_k(\Omega_k) \setminus
  \varphi_k(\widehat{J}_k)$ as $\varphi_k(\Omega_k)$ is an open set
  containing $\varphi_k(\widehat{J}_k)$.

  By Lemma~\ref{diff-compact}, the function
  \begin{equation*}
    \CauchyMod \psi{\varphi_k}{\widehat{J}_k}(\varphi_k(z)) - \Cauchy
    \psi{\widehat{J}_k}(z)
  \end{equation*}
  continues to a function in $A(\overline{\Omega}_k)$.  In particular,
  it is bounded in $\Omega_k \setminus \widehat{J}_k$.  Since $\Cauchy
  \psi{\widehat{J}_k}$ is bounded in $\C \setminus \widehat{J}_k$, it
  follows that $\CauchyMod \psi{\varphi_k}{\widehat{J}_k} \circ
  \varphi_k$ is bounded in $\Omega_k \setminus \widehat{J}_k$, or
  equivalently $\CauchyMod \psi{\varphi_k}{\widehat{J}_k}$ is bounded
  in $\varphi_k(\Omega_k \setminus \widehat{J}_k)$.  Since
  $\varphi_k(\Omega_k)\setminus\varphi_k(\widehat{J}_k) \subset
  \varphi_k(\Omega_k\setminus \widehat{J}_k)$, we conclude that
  $\CauchyMod \psi{\varphi_k}{\widehat{J}_k} \in
  H^\infty(\C\setminus\varphi_k(\widehat{J}_k))$.
\end{proof}

\begin{proof}[of Theorem~\ref{compact-operator}]
  Let us first justify that it is enough to prove the theorem for the
  case when each of the sets $J_k$ is a single arc and these arcs
  intersect only at their endpoints.  Write $J_k = \Gamma_{k,1} \cup
  \cdots \cup \Gamma_{k,r_k}$, where $\Gamma_{k,j}$ are disjoint arcs,
  and put $\psi_{k,j} = \varphi_k$, for $j=1,\ldots,r_k$.  Now pass to
  smaller arcs $\widetilde{\Gamma}_{k,j} \subset \Gamma_{k,j}$ such
  that the arcs $\widetilde{\Gamma}_{k,j}$ intersect only at endpoints
  but still cover all $\partial\Omega$.  The functions $\psi_{k,j}$
  and sets $\Gamma_{k,j}$ form an admissible family.  Assume that the
  conclusion of Theorem~\ref{compact-operator} is true for this
  family, and let $F_{k,j}$ be the linear operators associated to each
  of the functions $\psi_{k,j}$.  Putting $F_k =
  F_{k,1}+\cdots+F_{k,r_k}$ and recalling that $\psi_{k,j}=\varphi_k$,
  we see that the conclusion of Theorem~\ref{compact-operator} is true
  for the family $\{\varphi_k\}$ as well.

  We give the proof for a simply connected domain $\Omega$.  This case
  has the advantage that $\partial \Omega$ is a single Jordan curve,
  so the notation for numbering the arcs $J_k \subset \partial \Omega$
  is easier.  The proof for a multiply connected domain $\Omega$ is
  the essentially the same, except that the notation for the arcs
  $J_k$ is a bit more complex.

  Let us assume that the arcs $J_1,\ldots,J_n$ are numbered in a
  cyclic order, i.e., in such a way that $J_k$ intersects $J_{k-1}$
  and $J_{k+1}$ (here and henceforth we consider subindices modulo
  $n$).  Let $z_k \in \partial\Omega$ be the common endpoint of $J_k$
  and $J_{k+1}$, $k = 1,\ldots,n$.

  Let $V_k$ be a small disk centered at $z_k$ (its radius is
  determined later).  Choose functions
  $\eta_1,\ldots,\eta_n,\,\nu_1,\ldots,\nu_n \in
  \Cont^\infty(\partial\Omega)$ such that $0\leq\eta_k\leq 1$,
  $0\leq\nu_k\leq 1$ on $\partial\Omega$,
  $\eta_1+\cdots+\eta_n+\nu_1+\cdots+\nu_n = 1$, $\supp \nu_k \subset
  V_k \cap \partial\Omega$, and $\eta_k$ is supported on the interior
  of $J_k$ relative to $\partial\Omega$.

  \begin{figure}
    \begin{center}
      \includegraphics[width=7cm]{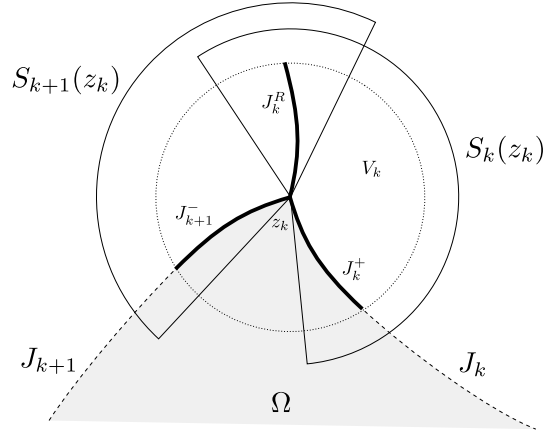}
      \caption{Geometric picture of the proof of
        Theorem~\ref{compact-operator}}
      \label{fig:1}
    \end{center}
  \end{figure}

  Put $J_k^+ = J_k \cap V_k$, $J_k^- = J_k \cap V_{k-1}$ and let $R_k$
  be a rigid rotation around the point $z_k$ such that $J_k^R
  \overset{\text{def}}{=} R_k J_k^+$ is contained in $[(S_k(z_k) \cap
  S_{k+1}(z_k))\setminus\overline{\Omega}]\cup \{z_k\}$ (see condition
  (d) in the definition of an admissible function).
  Figure~\ref{fig:1} is a picture of the relevant geometric objects.

  For $f \in H^\infty(\Omega)$, define
  \begin{equation}
    \label{eq:thm2-**}
    F_k(f) = \CauchyMod {f}{\varphi_k}{J_k}
    - \CauchyMod{(\nu_k f)\circ R_k^{-1}}{\varphi_k}{J_k^R} +
    \CauchyMod{(\nu_{k-1}f)\circ R_{k-1}^{-1}}{\varphi_k}{J_{k-1}^R}.
  \end{equation}
  Let us first check that $F_k(f) \in H^\infty(\D)$.  To do this, put
  \begin{equation}
    \begin{split}
      \label{eq:gk}
      G_k^+(f) &= \CauchyMod {\nu_k f}{\varphi_k}{J_k^+} -
      \CauchyMod {(\nu_kf)\circ R_k^{-1}}{\varphi_k}{J_k^R},\\
      G_k^-(f) &= \CauchyMod {\nu_k f}{\varphi_{k+1}}{J_{k+1}^-} +
      \CauchyMod {(\nu_kf)\circ R_k^{-1}}{\varphi_{k+1}}{J_k^R}.
    \end{split}
  \end{equation}
  Then we can write $F_k(f)$ as
  \begin{equation*}
    F_k(f) = \CauchyMod {\eta_k f}{\varphi_k}{J_k} + G_k^+(f) +
    G_{k-1}^-(f),
  \end{equation*}
  because
  \begin{equation*}
    \CauchyMod {f}{\varphi_k}{J_k} = \CauchyMod {\eta_k
      f}{\varphi_k}{J_k} + \CauchyMod {\nu_k f}{\varphi_k}{J_k^+} +
    \CauchyMod {\nu_{k-1} f}{\varphi_k}{J_k^-}.
  \end{equation*}
  Since $f \in H^\infty(\Omega)$, and $\eta_k$ is supported on a
  closed arc contained in the interior of $J_k$, it is easy to see
  that $\Cauchy {\eta_k f}{J_k}$ belongs to $H^\infty(\C\setminus
  J_k)$.  Lemma~\ref{lemma-simple} allows us to conclude that
  $\CauchyMod {\eta_k f}{\varphi_k}{J_k}$ belongs to $H^\infty(\C
  \setminus \varphi_k(J_k))$.  As $\varphi_k(J_k)\subset \T$, this
  implies that $\CauchyMod {\eta_k f}{\varphi_k}{J_k} \in
  H^\infty(\D)$.

  Since $|\varphi_k| < 1$ in $\Omega$ and $|\varphi_k| = 1$ in $J_k$,
  by the Schwarz reflection principle we can assume that
  \begin{equation*}
    |\varphi_k| > 1 \text{ in } \Omega_k \setminus \Omega
  \end{equation*}
  just by making $\Omega_k$ smaller if necessary (i.e., replacing
  $\Omega_k$ by $U_k\cap\Omega_k$, where $U_k$ is some open set
  containing $J_k \cup \Omega$).

  The following claim is justified below.
  \begin{claim}
    \label{claim1}
    $G_k^+(f)\in H^\infty(\C\setminus\varphi_k(J_k^+ \cup J_k^R))$ and
    $G_k^-(f)\in H^\infty(\C\setminus\varphi_{k+1}(J_{k+1}^- \cup
    J_k^R))$.
  \end{claim}

  Since $|\varphi_k| = 1$ in $J_k$, $\varphi_k(J_k^+\cup J_k^R) \cap
  \D = \emptyset$, and so by the claim, $G_k^+(f)$ and $G_k^-(f)$
  belong to $H^\infty(\Omega)$.  It follows that $F_k(f) \in
  H^\infty(\D)$ for every $f \in H^\infty(\Omega)$.  Moreover, it is
  clear from the proof of these lemmas that $F_k$ maps
  $H^\infty(\Omega)$ into $H^\infty(\D)$ and is bounded.

  We next show that the linear map
  \begin{equation*}
    f \mapsto f - \sum_{k=1}^n F_k(f) \circ \varphi_k,
  \end{equation*}
  is a compact operator on $H^\infty(\Omega)$, whose range is
  contained in $\A(\overline{\Omega})$.  A simple calculation using
  \begin{equation*}
    f = \sum_{k=1}^n \Cauchy f {J_k}
  \end{equation*}
  gives
  \begin{equation}
    \label{eq:thm2-**2}
    f - \sum_{k=1}^n F_k(f) \circ \varphi_k = \sum_{k=1}^n A_k(f) +
    B_k(\nu_k f),
  \end{equation}
  where
  \begin{equation}
    \label{eq:AkBk}
    \begin{split}
      A_k(\psi) &= \Cauchy {\psi}{J_k} - \CauchyMod
      {\psi}{\varphi_k}{J_k} \circ \varphi_k,\\
      B_k(\psi) &= \CauchyMod {\psi \circ
        R_k^{-1}}{\varphi_{k+1}}{J_k^R} \circ \varphi_{k+1} -
      \CauchyMod {\psi \circ R_k^{-1}}{\varphi_k}{J_k^R} \circ
      \varphi_k.
    \end{split}
  \end{equation}
  By Lemma~\ref{diff-compact}, the operator $A_k$ is compact from
  $L^\infty(\partial\Omega)$ into $A(\overline{\Omega})$.  To see that
  $B_k$ has the same property, write
  \begin{equation*}
    B_k(\psi) = \big[\CauchyMod {\psi \circ
      R_k^{-1}}{\varphi_{k+1}}{J_k^R} \circ \varphi_{k+1} - \Cauchy
    {\psi \circ R_k^{-1}}{J_k^R} \big] + \big[ \Cauchy {\psi \circ
      R_k^{-1}}{J_k^R} - \CauchyMod {\psi \circ
      R_k^{-1}}{\varphi_k}{J_k^R} \circ \varphi_k \big],
  \end{equation*}
  and apply Lemma~\ref{diff-compact} to each of the two terms in
  brackets.

  It remains to prove that the operators $F_k$ map
  $A(\overline{\Omega})$ into $A(\overline{\D})$.  It is enough to
  check that if $f$ is analytic on some open neighborhood of
  $\overline{\Omega}$, then $F_k(f) \in \Cont(\overline{\D})$, as the
  space of functions analytic on $\overline{\Omega}$ is dense in
  $A(\overline{\Omega})$ and $F_k$ is bounded.

  By~\eqref{eq:thm2-**} and properties of the modified Cauchy
  integral, $F_k(f)$ is continuous on
  $\overline{\D}\setminus\varphi_k(J_k)$.  Next check that $F_k(f)$
  extends by continuity to $\varphi_k(J_k)$.  Since $\varphi_k'$ does
  not vanish on $J_k$, there exists a continuous local inverse of
  $\varphi_k$ on each point of $\varphi_k(J_k)$.  This implies that it
  is enough to verify that $F_k(f)\circ \varphi_k$ is continuous in
  $\overline{\Omega}$.  Put
  \begin{equation*}
    \begin{split}
      \widetilde{F}_k(f) &= \Cauchy {\eta_k f}{J_k} +
      \widetilde{G}_k^+(f) + \widetilde{G}_{k-1}^-(f),\\
      \widetilde{G}_k^+(f) &= \Cauchy {\nu_k f}{J_k^+} -
      \Cauchy {(\nu_kf)\circ R_k^{-1}}{J_k^R},\\
      \widetilde{G}_k^-(f) &= \Cauchy {\nu_k f}{J_{k+1}^-} + \Cauchy
      {(\nu_kf)\circ R_k^{-1}}{J_k^R},
    \end{split}
  \end{equation*}
  i.e., replace the modified Cauchy integrals in the formulas for
  $F_k$, $G_k^-$ and $G_k^+$ by regular Cauchy integrals to get
  $\widetilde{F}_k$, $\widetilde{G}_k^-$ and $\widetilde{G}_k^+$.
  Arguing as above for the operators $A_k$ and $B_k$, we see that $f
  \mapsto F_k(f)\circ\varphi_k - \widetilde{F}_k(f)$ defines a compact
  operator whose range is contained in $\Cont(\overline{\Omega})$.
  Thus it is enough to show that $\widetilde{F}_k(f) \in
  \Cont(\overline{\Omega})$.

  By Lemma~\ref{lemma-A} below, it is easy to see that $\Cauchy
  {\eta_kf}{J_k} \in \Cont(\overline{\Omega})$.  We have
  \begin{equation*}
    \widetilde{G}_k^+(f) + \widetilde{G}_k^-(f) = \Cauchy
    {\nu_kf}{\partial\Omega\cap V_k}.
  \end{equation*}
  Also by Lemma~\ref{lemma-A}, the right hand side of this equality
  belongs to $\Cont(\overline{\Omega})$.  Therefore, it suffices to
  check that $\widetilde{G}_k^-(f) \in \Cont(\overline{\Omega})$.

  Now $\widetilde{G}_k^-(f) = \Cauchy {\widetilde{f}} {J_{k+1}^-\cup
    J_k^R}$, where $\widetilde{f}(z) = (\nu_k f)(z)$ for $z \in
  J_{k+1}^-$, and $\widetilde{f}(z) = (\nu_k f)(R_k^{-1}(z))$ for $z
  \in J_k^R$.  Since $f$ is analytic in a neighborhood of
  $\overline{\Omega}$, $\widetilde{f}$ is Lipschitz in $J_{k-1}^- \cup
  J_k^R$, and since $\widetilde{f}$ vanishes identically near the
  endpoints of $J_{k-1}^- \cup J_k^R$, Lemma~\ref{lemma-A} implies
  that $\widetilde{G}_k^-(f) \in \Cont(\overline{\Omega})$.  This
  finishes the proof of the theorem.
\end{proof}

\begin{proof}[of Claim~\ref{claim1}]
  We use the same techniques as those used in~\cite{HavinNersessian}
  to prove Theorem~4.1 to show that $g_k^-
  \overset{\text{def}}{=}G_k^-(f) \in
  H^\infty(\C\setminus\varphi_{k+1}(J_{k+1}^- \cup J_k^R))$.  Similar
  reasoning can be applied to $G_k^+(f)$.

  Let
  \begin{equation*}
    h_k^+ = \Cauchy {\nu_k f}{J_k^+},\qquad h_k^- = \Cauchy {\nu_k
      f}{J_{k+1}^-},
  \end{equation*}
  so that $h_k^- + h_k^+ = \Cauchy {\nu_k f}{V_k\cap\partial\Omega}$,
  which because $f \in H^\infty(\Omega)$, belongs to
  $H^\infty(\C\setminus(V_k\cap\partial\Omega))$.  Theorem 4.1
  in~\cite{HavinNersessian} applies, and so $h_k^- + h_k^+ \circ
  R_k^{-1}$ belongs to $H^\infty(\C\setminus(J_k^- \cup J_k^R))$.

  We next prove that $g_k^-$ is bounded in $\C \setminus
  \varphi_{k+1}(J_{k+1}^- \cup J_k^R)$.  It is clearly analytic in
  this set.  Let $S_{k+1}^-$ be an open circular sector with vertex on
  $z_k$, such that $J_{k+1}^- \cup J_k^R \subset S_{k+1}^- \cup
  \{z_k\}$ and $S_{k+1}^- \subset \Omega_{k+1}$.  This circular sector
  can be chosen by shrinking one of the circular sectors which appear
  in condition (d) in the definition of an admissible family.  We
  first show that $g_k^-$ is bounded in
  $\varphi_{k+1}(S_{k+1}^-\setminus(J_{k+1}^- \cup J_k^R))$.  To do
  this, observe that by a change of variables in the integral defining
  the Cauchy transform,
  \begin{equation*}
    h_k^+ \circ R_k^{-1} = \Cauchy {(\nu_k f)\circ R_k^{-1}}{J_k^R}.
  \end{equation*}
  Now compute
  \begin{equation*}
    \begin{split}
      g_k^- \circ \varphi_{k+1} - (h_k^- + h_k^+ \circ R_k^{-1}) &=
      \big[\CauchyMod {\nu_kf}{\varphi_{k+1}}{J_{k+1}^-} \circ
      \varphi_{k+1} - \Cauchy{\nu_kf}{J_{k+1}^-}\big] \\ &\quad +
      \big[\CauchyMod {(\nu_kf)\circ R_k^{-1}}{\varphi_{k+1}}{J_k^R}
      \circ \varphi_{k+1} - \Cauchy {(\nu_kf)\circ R_k^{-1}}{J_k^R}
      \big].
    \end{split}
  \end{equation*}
  A similar argument to the one used in Lemma~\ref{lemma-simple} shows
  that each of the expressions in square brackets is bounded in
  $S_{k+1}^- \setminus (J_{k+1}^- \cup J_k^R)$.  Therefore, $g_k^-
  \circ \varphi_{k+1}$ is also bounded in this set as $h_k^- + h_k^+
  \circ R_k^{-1}$ is bounded there.  It follows that $g_k^-$ is
  bounded in $\varphi_{k+1}(S_{k+1}^-\setminus(J_{k+1}^- \cup
  J_k^R))$.

  \begin{figure}
    \begin{center}
      \includegraphics[width=7cm]{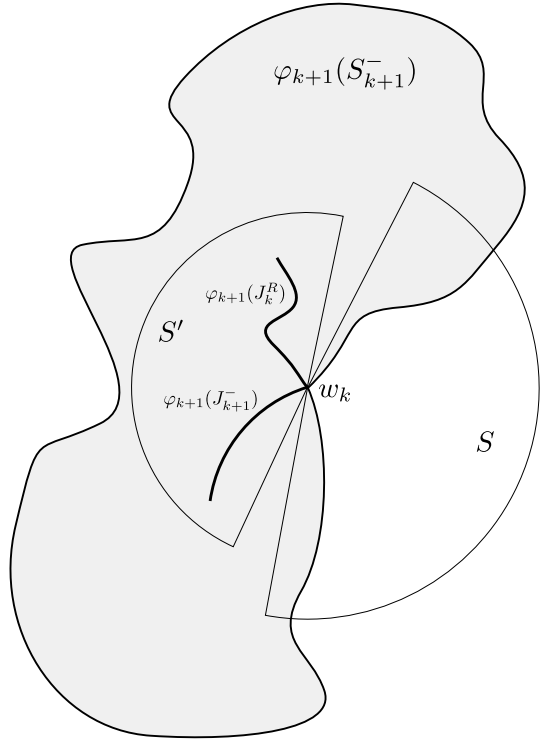}
      \caption{The circular sectors $S$ and $S'$}
      \label{fig:2}
    \end{center}
  \end{figure}

  It remains to prove that $g_k^-$ is bounded in $\C \setminus
  \varphi_{k+1}(S_{k+1}^-)$.  If $S$ is an open circular sector, we
  say that its straight edges are the two line segments which form a
  part of its boundary.  Put $w_k = \varphi_{k+1}(z_k)$.  Choose two
  open circular sectors $S$ and $S'$ with vertex $w_k$ having the
  following properties (see Figure~\ref{fig:2}):
  \begin{itemize}
  \item $\overline{S} \cap \overline{S'} = \{w_k\}$.
  \item $\varphi_{k+1}(J_{k+1}^- \cup J_k^R) \subset S' \cup \{w_k\}$.
  \item $\D_\varepsilon(w_k) \setminus \varphi_{k+1}(S_{k+1}^-)
    \subset S \cup \{w_k\}$ for some $\varepsilon > 0$.
  \item The straight edges of $S$ are contained in
    $\varphi_{k+1}(S_{k+1}^-) \cup \{w_k\}$.
  \end{itemize}
  Such circular sectors can be chosen by shrinking $V_k$ if necessary,
  using the fact that $\varphi_{k+1}$ is conformal at $z_k$.

  It is enough to show that $g_k^-$ is bounded in $S$, because
  $g_k^-(z)$ is clearly uniformly bounded when $z$ is away from
  $\varphi_{k+1}(J_{k+1}^- \cup J_k^R)$, and we have already seen that
  $g_k^-$ is bounded in $\varphi_{k+1}(S_{k+1}^- \setminus (J_{k+1}^-
  \cup J_k^R))$.  This is done by using a weak form of the
  Phragm\'en-Lindel\"of principle, in the same manner as
  in~\cite{HavinNersessian}.  Since $g_k^-$ is bounded in the straight
  edges of $S$ except at the vertex $w_k$ (the straight edges are
  contained in $\varphi_{k+1}(S_{k+1}^-\setminus(J_{k+1}^- \cup
  J_k^R))$, except for $w_k$), it suffices to show that $g_k^-$ is
  $O(|z-w_k|^{-1})$ as $z \to w_k$, $z \in S$.

  First, estimate
  \begin{equation*}
    \begin{split}
      |(z-w_k)g_k^-(z)| &\leq \int_{J_{k+1}^-}
      \frac{|z-w_k|}{|\varphi_{k+1}(\zeta) -
        z|}|(\nu_kf)(\zeta)\varphi_{k+1}'(\zeta)|\,|d\zeta|
      \\
      &+ \int_{J_k^R} \frac{|z-w_k|}{|\varphi_{k+1}(\zeta) -
        z|}|(\nu_kf)(\zeta)\varphi_{k+1}'(\zeta)|\,|d\zeta|
      \\
      &\leq \|f\|_\infty\int_{J_{k+1}^-\cup J_k^R}
      \frac{|z-w_k|}{|\varphi_{k+1}(\zeta) -
        z|}|\varphi_{k+1}'(\zeta)|\,|d\zeta|.
    \end{split}
  \end{equation*}
  We claim that $a(\zeta,z) \overset{\text{def}}{=}
  |z-w_k|/|\varphi_{k+1}(\zeta)-z|$ is uniformly bounded for $z \in S$
  and $\zeta \in J_k^- \cup J_k^R$, which follows from the observation
  that $\varphi_{k+1}(\zeta) \in S'$ and $z \in S$, so that
  $a(\zeta,z) \leq C$ due to the geometry of the cones $S$ and $S'$.
  The last integral is therefore uniformly bounded, so $g_k^-$ is
  $O(|z - w_k|^{-1})$ and we conclude that $g_k^-$ belongs to
  $H^\infty(\C\setminus\varphi_{k+1}(J_{k+1}\cup J_k^R))$.

  To see that $g_k^+$ belongs to
  $H^\infty(\C\setminus\varphi_{k}(J_k^+\cup J_k^R))$, use similar
  reasoning with $h_k^+ - h_k^+ \circ R_k^{-1}$ instead of $h_k^- +
  h_k^+ \circ R_k^{-1}$, an appropriate circular sector $S_k^+$ for
  $S_{k+1}^-$, and $\varphi_k$ in place of~$\varphi_{k+1}$.
\end{proof}

The following lemma is well known from the classical theory of Cauchy
integrals.  See, for instance,~\cite{Gajov}*{Chapter I, Section 5.1}.

\begin{lemma}
  \label{lemma-A}
  Let $\Gamma$ be a piecewise smooth Jordan curve and $\Omega$ the
  region interior to it.  If $\psi$ is of class H\"older $\alpha$ on
  $\Gamma$, $0<\alpha<1$, then $\Cauchy \psi \Gamma$ is of class
  H\"older $\alpha$ in $\overline{\Omega}$.
\end{lemma}

\section{Weak* closedness}
\label{weak-star}

In this section we prove that the algebra $\mathcal{H}_\Phi$ is
weak*-closed in $H^\infty$.  First recall a well known result about
weak*-continuity of adjoint operators, the proof of which is
elementary and so is omitted.

\begin{lemma}
  If $X$ is a Banach space and $T: X \to X$ is a bounded operator,
  then its adjoint $T^*: X^* \to X^*$ is continuous in the
  weak*-topology of $X^*$.
\end{lemma}

The operator $T$ is called the predual of $T^*$.  Thus any operator
with a predual is weak*-continuous, a condition applying to many
integral operators on $L^\infty$.

\begin{lemma}
  \label{predual}
  Let $T : L^\infty(\partial\Omega) \to L^\infty(\partial\Omega)$ be
  defined by
  \begin{equation*}
    (Tf)(z) = \int_{\partial\Omega} G(\zeta,z)f(\zeta)\,d\zeta,
  \end{equation*}
  where $G : \partial\Omega \times \partial\Omega \to \C$ is a
  measurable function satisfying
  \begin{equation*}
    \int_{\partial\Omega} |G(\zeta,z)|\,|d\zeta| \leq C
  \end{equation*}
  for every $z \in \partial\Omega$.  Then the operator $S$ defined by
  \begin{equation*}
    (Sg)(\zeta) = \int_{\partial\Omega} G(\zeta,z)g(z)\,dz
  \end{equation*}
  is a bounded operator $S:L^1(\partial\Omega)\to L^1(\partial\Omega)$
  and satisfies $S^* = T$.
\end{lemma}

\begin{proof}
  Fubini's Theorem shows that
  \begin{equation*}
    \left|\int_{\partial\Omega} (Sg)(\zeta) f(\zeta)\, d\zeta \right|
    \leq C\|g\|_1\|f\|_\infty,
  \end{equation*}
  and so $S$ is bounded on $L^1(\partial\Omega)$.  Another application
  of Fubini's Theorem gives $S^* = T$.
\end{proof}

Recall from~\eqref{eq:thm2-**2} and the proof of
Theorem~\ref{compact-operator} that the operator
\begin{equation*}
  K(f) = f - \sum_{k=1}^n F_k(f) \circ \varphi_k,
\end{equation*}
is a weakly singular integral operator of the form
\begin{equation*}
  K(f)(z)\mapsto \int_{\partial\Omega} G(\zeta,z) f(\zeta)\, d\zeta,
\end{equation*}
where the function $G$ is continuous outside the diagonal $\{\zeta =
z\}$ and $|G(\zeta,z)|\leq C|\zeta-z|^{-\beta}$, for some $\beta < 1$.
Also, for each $\zeta \in \partial\Omega$, $G(\zeta,z)$ is analytic in
$z \in \Omega$.  Thus, $K$ is compact from $L^\infty(\Omega)$ to
$H^\infty(\Omega)$.  By Lemma~\ref{predual}, the operator $K$ has a
predual, so is weak*-continuous.

\begin{lemma}
  For every $\varepsilon > 0$, there is an operator $K_\varepsilon :
  L^\infty(\partial\Omega) \to L^\infty(\partial\Omega)$ of finite
  rank which has a predual and such that $\|K_\varepsilon - K\| <
  \varepsilon$.
\end{lemma}

\begin{proof}
  Fix $\varepsilon > 0$.  Since $G$ is continuous outside $\{\zeta =
  z\}$, there exist $\alpha_j \in L^\infty(\partial\Omega)$ and
  $\beta_j \in L^1(\partial\Omega)$ such that
  \begin{equation*}
    \int_{\partial\Omega} \Big|G(\zeta,z) - \sum_{j=1}^N
    \alpha_j(z)\beta_j(\zeta)\Big|\,d\zeta < \varepsilon/2
  \end{equation*}
  for every $z \in \partial\Omega$.  This implies that the finite rank
  operator $K_\varepsilon$ defined by
  \begin{equation*}
    K_\varepsilon(\psi)(z) = \int_{\partial\Omega} \sum_{j=1}^N
    \alpha_j(z)\beta_j(\zeta)\psi(\zeta)\, d\zeta
  \end{equation*}
  satisfies $\|K_\varepsilon - K\| \leq \varepsilon/2$.  Clearly, by
  Lemma~\ref{predual}, the operator $K_\varepsilon$ has a predual.
\end{proof}

\begin{lemma}
  \label{weak-star-closed}
  Let $L(f) = \sum_{k=1}^n F_k(f)\circ\varphi_k$ be the operator $L :
  H^\infty(\Omega) \to H^\infty(\Omega)$ of
  Theorem~\ref{compact-operator}.  Then the range of $L$ is
  weak*-closed in $H^\infty(\Omega)$.
\end{lemma}

\begin{proof}
  We have $L = (I-K)|H^\infty(\Omega)$.  By the preceding lemma with
  $\varepsilon = 1$, there is a finite rank operator $K_1$ such that
  $\|K_1 - K\| < 1$.  Put $M = H^\infty(\Omega) +
  K_1(L^\infty(\partial\Omega))$.  Then, since $H^\infty(\Omega)$ is
  weak*-closed in $L^\infty(\partial\Omega)$, $M$ is a weak*-closed
  subset of $L^\infty(\Omega)$ such that $H^\infty(\Omega)$ has finite
  codimension in $M$.  Define $\Delta = I - (K-K_1)$.  Note that
  $K_1(L^\infty(\partial\Omega)) \subset M$ and
  $K(L^\infty(\partial\Omega)) \subset H^\infty(\Omega) \subset M$.
  Since
  \begin{equation*}
    \Delta^{-1} = \sum_{j=0}^\infty (K-K_1)^j,
  \end{equation*}
  this series being convergent in operator norm, we also have
  $\Delta^{-1}M \subset M$.

  Now observe that
  \begin{equation*}
    L(H^\infty(\Omega)) = (I - K)H^\infty(\Omega) = \Delta(I -
    \Delta^{-1}K_1)H^\infty(\Omega).
  \end{equation*}
  Put $X = (I-\Delta^{-1}K_1)H^\infty(\Omega)$ and note that $\ker K_1
  \cap H^\infty(\Omega) \subset X$.  Since $\ker K_1 \cap
  H^\infty(\Omega)$ is weak*-closed and has finite codimension in $M$,
  and $X \subset M$, it follows that $X$ is weak* closed.

  It remains to show that $\Delta X$ is weak*-closed.  It is enough to
  check that $\Delta^{-1}$ is weak*-continuous.  Since $K$ and $K_1$
  have preduals, it follows that $\Delta$ has a predual.  Therefore,
  $\Delta^{-1}$ also has a predual, and so it is weak*-continuous.
\end{proof}

Finally, we can show that $\mathcal{H}_\Phi$ is weak*-closed in
$H^\infty(\Omega)$.  The argument is similar to the proof of
Lemma~\ref{finite-codim}.

\begin{lemma}
  If $\Phi : \overline{\Omega}\to\overline{\D}^n$ is admissible, then
  $\mathcal{H}_\Phi$ is weak*-closed in $H^\infty(\Omega)$.
\end{lemma}

\begin{proof}
  We have already seen in the proof of Lemma~\ref{finite-codim} that
  the range of the operator $f \mapsto \sum F_k(f) \circ \varphi_k$,
  $f \in H^\infty(\Omega)$, has finite codimension in
  $H^\infty(\Omega)$.  By the preceding lemma, the range is also
  weak*-closed.  Since $\mathcal{H}_\Phi$ contains this range, we get
  that $\mathcal{H}_\Phi$ is weak*-closed.
\end{proof}

\section{Glued subalgebras}
\label{glued-subalgebras}

In this section we characterize the maximal ideal space and the
derivations of finite codimensional subalgebras of a unital
commutative Banach algebra.  The arguments used are purely algebraic
and similar results hold for arbitrary unital commutative complex
algebras.  In the algebraic setting, one should replace the maximal
ideal space by the set of all (unital) homomorphisms of the algebra
into the complex field and disregard every reference made to the
topology, such as closed subspaces and continuity of homomorphisms and
derivations.

Let $A$ be a commutative unital Banach algebra.  A \emph{glued
  subalgebra} of $A$ is understood to be a (unital) subalgebra of the
form
\begin{equation}
  \label{eq:glued-subalgebra}
  B = \{f \in A : \alpha_j(f) = \beta_j(f),\ j = 1,\ldots,r\},
\end{equation}
where $\alpha_j,\beta_j \in \M(A)$ and $\alpha_j \neq \beta_j$ for
$j=1,\ldots,r$.  We define the set of points of $A$ glued in $B$ as
\begin{equation*}
  G(A,B) = \{\alpha_j : j = 1,\ldots,r \} \cup \{\beta_j : j =
  1,\ldots,r\} \subset \M(A).
\end{equation*}

Our first goal is to characterize the space $\M(B)$ in terms of
$\M(A)$.  Since $B$ is a subalgebra of $A$, there is a map $i^*:
\M(A)\to\M(B)$ which sends each complex homomorphism $\psi \in \M(A)$
to its restriction $\psi|B \in \M(B)$.  We first show that $i^*$ is
onto.  To do this, we need to use the so called ``lying over lemma'',
which applies to integral ring extensions.

Recall that if $R$ is a subring of some ring $S$, then $S$ is called
\emph{integral} over $R$ if for every $\alpha\in S$ there is a monic
polynomial $p \in R[x]$ such that $p(\alpha) = 0$.  It is well known
that if $B$ is a finite codimensional subalgebra of some algebra $A$,
then $A$ is integral over $B$.

The following ``lying over lemma'' or Cohen-Seidenberg theorem is a
standard result from commutative algebra.  It was originally proved in
\cite{CohenSeidenberg}.

\begin{lemma}[(Lying over lemma)]
  If $S$ is integral over $R$ and $P$ is a prime ideal in $R$, then
  there is a prime ideal $Q$ in $S$ such that $P = Q\cap R$ $($we say
  that $Q$ is \emph{lying over} $P)$.  If $Q$ is a prime ideal in $S$
  lying over $P$, then $Q$ is maximal if and only if $P$ is maximal.
\end{lemma}

\begin{lemma}
  \label{glued-onto}
  If $B$ is a finite codimensional closed unital subalgebra of a
  commutative unital Banach algebra $A$, then $i^*:\M(A)\to\M(B)$ is
  onto.
\end{lemma}

\begin{proof}
  Let $\psi_B\in\M(B)$ and put $P = \ker \psi_B$.  Then $P$ is a
  maximal ideal in $B$.  By the lying over lemma, there is a maximal
  ideal $Q$ in $A$ such that $Q\cap B = P$.  Since $P$ is closed and
  has finite codimension in $A$ and $Q \supset P$, it follows that $Q$
  is closed.  Every maximal ideal in a Banach algebra has codimension
  one, so $Q$ has codimension one in $A$.  Therefore, there exists a
  unique $\psi_A\in\M(A)$ such that $\ker \psi_A = Q$.  The equality
  $Q\cap B = P$ implies $i^*(\psi_A) = \psi_B$.
\end{proof}

Note that since $\M(A)$ is compact and $i^*$ is continuous, it follows
that $i^*$ is topologically a quotient map.

In the purely algebraic setting, it is no longer true that every
maximal ideal has codimension one.  However, one can still show that
$Q$ has codimension one in $A$.  Indeed, note that $P$ has finite
codimension in $B$.  Therefore $P$ has finite codimension in $A$.
Since $P \subset Q$, it follows that $Q$ has finite codimension in
$A$.  Now, $A/Q$ is a field, because $Q$ is a maximal ideal in $A$.
Also, $A/Q$ is a finite dimensional vector space over $\C$.  Since
finite field extensions are algebraic and $\C$ is algebraically
closed, it follows that $A/Q$ is isomorphic to $\C$, so that $Q$ has
codimension one in $A$.

The following is a kind of ``interpolation'' lemma.  It will be very
useful in the sequel~\cite{Article2}.

\begin{lemma}
  \label{glued-interpolation}
  Let $A$ be a commutative unital Banach algebra.  If
  $\psi_0,\ldots,\psi_s$ are distinct points in $\M(A)$, then there is
  some $f \in A$ such that $\psi_j(f) = 0$ for $j = 1,\ldots,s$, but
  $\psi_0(f) = 1$.
\end{lemma}

\begin{proof}
  Fix $j\in \{1,\ldots,s\}$.  There is some $f_j\in A$ such that
  $\psi_0(f_j) \neq 0$ and $\psi_j(f_j) = 0$, for if this were not the
  case, then $\ker \psi_0 \subset \ker \psi_j$, which would imply that
  $\ker \psi_0 = \ker \psi_j$ since both kernels have codimension one
  in $A$.  Hence we would have $\psi_0 = \psi_j$, a contradiction.

  For $f_1,\ldots,f_s$ chosen in this way,
  \begin{equation*}
    f = \prod_{j=1}^s \frac{f_j}{\psi_0(f_j)}.
  \end{equation*}
  has the required properties.
\end{proof}

\begin{lemma}
  \label{glued-preimage}
  If $B$ is a glued subalgebra of $A$ and $\psi_B\in\M(B)$, then
  either $(i^*)^{-1}(\{\psi_B\}) \subset G(A,B)$ or
  $(i^*)^{-1}(\{\psi_B\}) = \{\psi_A\}$, for some $\psi_A \notin
  G(A,B)$.
\end{lemma}

\begin{proof}
  Assume that we have distinct $\psi_A, \widetilde{\psi}_A \in
  (i^*)^{-1}(\{\psi_B\})$ with $\psi_A \notin G(A,B)$.  By
  Lemma~\ref{glued-interpolation}, there is an $f \in A$ such that
  $\psi_A(f) = 1$ and $\psi(f) = 0$ for $\psi \in G(A,B) \cup
  \{\widetilde{\psi}_A\}$, as $\psi_A \notin
  G(A,B)\cup\{\widetilde{\psi}_A\}$ by hypothesis.  Then $f \in B$,
  because $\alpha_j(f) = \beta_j(f) = 0$, for $j=1,\ldots,r$, and
  \begin{equation*}
    1 = \psi_A(f) = \psi_B(f) = \widetilde{\psi}_A(f) = 0,
  \end{equation*}
  because $\psi_A|B = \psi_B = \widetilde{\psi}_A|B$.  This is a
  contradiction.
\end{proof}

We next describe the derivations of $B$ in terms of the derivations
of~$A$.  This requires the following well-known characterization of
derivations: \emph{A linear functional $\eta$ on $A$ is a derivation
  at $\psi \in \M(A)$ if and only if $\eta(1) = 0$ and $\eta(fg) = 0$
  whenever $f,g\in A$ and $\psi(f) = \psi(g) = 0$.}

\begin{lemma}
  \label{glued-derivations}
  Let $B$ be a glued subalgebra of $A$, and $\eta_B$ a derivation of
  $B$ at a point $\psi_B \in \M(B)$.  Put
  \begin{equation*}
    \{\psi_A^1,\ldots,\psi_A^s\} = (i^*)^{-1}(\{\psi_B\}) \subset \M(A)
  \end{equation*}
  $($this set is finite by Lemma~\ref{glued-preimage}$)$.  Then there
  exist unique derivations $\eta_A^1,\ldots,\eta_A^s$ of $A$ at the
  points $\psi_A^1,\ldots,\psi_A^s$ respectively such that
  \begin{equation*}
    \eta_B = (\eta_A^1 + \ldots + \eta_A^s)|B.
  \end{equation*}
\end{lemma}

\begin{proof}
  For each $k = 1,\ldots,s$, use Lemma~\ref{glued-interpolation} to
  obtain an $f_k \in A$ such that $\psi_A^j(f_k) = \delta_{jk}$ and
  $\psi(f_k) = 0$ for $\psi \in G(A,B) \setminus
  \{\psi_A^1,\ldots,\psi_A^s\}$.  Put $g_k = 2f_k - f_k^2$.  Then
  $g_k$ also satisfies $\psi_A^j(g_k) = \delta_{jk}$ and $\psi(g_k) =
  0$ for $\psi \in G(A,B)\setminus\{\psi_A^1,\ldots,\psi_A^s\}$.

  Define $\eta_A^j$ by
  \begin{equation*}
    \eta_A^j(f) = \eta_B(g_j^2(f-\psi_A^j(f))),\qquad f\in A.
  \end{equation*}
  Since $\psi(g_j^2(f-\psi_A^j(f))) = 0$ for every $\psi \in G(A,B)$,
  $g_j^2(f-\psi_A^j(f)) \in B$ and so $\eta_A^j$ is well defined.

  We claim that $\eta_A^j$ is a derivation of $A$ at $\psi_A^j$.  It
  is clearly linear and $\eta_A^j(1) = 0$.  Take $f,g \in A$ with
  $\psi_A^j(f) = \psi_A^j(g) = 0$.  Then,
  \begin{equation*}
    \eta_A^j(fg) = \eta_B(g_j^2fg) = 0,
  \end{equation*}
  because $g_jf$ and $g_jg$ belong both to $B$ and $\psi_B(g_jf) =
  \psi_B(g_jg) = 0$ as $\psi(g_jf) = \psi(g_jg) = 0$ for every $\psi
  \in G(A,B) \cup \{\psi_A^1,\ldots\psi_A^s\}$.  It follows that
  $\eta_A^j$ is a derivation of $A$ at $\psi_A^j$.

  Now we check that $(\eta^1_A + \cdots + \eta_A^s)|B = \eta_B$.  For
  this, put
  \begin{equation*}
    g_0 = g_1^2 + \cdots + g_s^2.
  \end{equation*}
  Note that $\psi_A^1(g_0) = \cdots = \psi_A^s(g_0) = 1$ and
  $\psi(g_0) = 0$ for $\psi \in
  G(A,B)\setminus\{\psi_A^1,\ldots,\psi_A^s\}$.  If $\alpha_j,\beta_j$
  are as in~\eqref{eq:glued-subalgebra}, then $i^*(\alpha_j) =
  i^*(\beta_j)$.  Therefore, either $\alpha_j,\beta_j$ both belong to
  $(i^*)^{-1}(\{\psi_B\}) = \{\psi_A^1,\ldots,\psi_A^s\}$ or they both
  belong to $G(A,B)\setminus\{\psi_A^1,\ldots,\psi_A^s\}$.  Hence,
  $\alpha_j(g_0) = \beta_j(g_0)$, because $\alpha_j(g_0)$ and
  $\beta_j(g_0)$ are both $1$ or both $0$.  Therefore, $g_0 \in B$.
  Also, $\psi_B(g_0) = \psi_A^1(g_0) = 1$.

  Take any $f \in B$.  Then
  \begin{equation*}
    \sum_{j=1}^s \eta_A^j(f) = \sum_{j=1}^s \eta_B(g_j^2(f-\psi_A^j(f)))
    = \eta_B(g_0(f-\psi_B(f))) = \eta_B(f-\psi_B(f)) = \eta_B(f),
  \end{equation*}
  because $\psi_B(f-\psi_B(f)) = 0$, and $\psi_B(g_0) = 1$.  This
  shows that $(\eta_A^1+\cdots+\eta_A^s)|B = \eta_B$.

  To prove uniqueness, assume that $\widetilde{\eta}_A^1, \ldots,
  \widetilde{\eta}_A^s$ are derivations of $A$ at
  $\psi_A^1,\ldots,\psi_A^s$ respectively and such that
  $(\widetilde{\eta}_A^1+\cdots+\widetilde{\eta}_A^s)|B = \eta_B$.

  Since $\psi_A^j(g_j) = \psi_A^j(f_j) = 1$,
  \begin{equation*}
    \widetilde{\eta}_A^j(g_j^2) =
    2\widetilde{\eta}_A^j(g_j)\psi_A^j(g_j) = 2\widetilde{\eta}_A^j(g_j)
    = 2\widetilde{\eta}_A^j(2f_j-f_j^2) = 4\widetilde{\eta}_A^j(f_j) -
    4\widetilde{\eta}_A^j(f_j)\psi_A^j(f_j) = 0,
  \end{equation*}
  and so $\widetilde{\eta}_A^j(g_j^2) = 0$.  Thus for any $f \in A$,
  \begin{equation*}
    \widetilde{\eta}_A^j(f) = \widetilde{\eta}_A^j(g_j^2f) =
    \widetilde{\eta}_A^j(g_j^2(f-\psi_A^j(f))),
  \end{equation*}
  and if $j\neq k$, then
  \begin{equation*}
    \widetilde{\eta}_A^k(g_j^2(f-\psi_A^j(f))) = 0,
  \end{equation*}
  because $\psi_A^k(g_j) = \psi_A^k(g_j(f-\psi_A^j(f))) = 0$.  Also,
  since $\psi(g_j^2(f-\psi_j(f))) = 0$ for every $\psi \in G(A,B)$,
  $g_j^2(f-\psi_j(f)) \in B$.  Hence,
  \begin{equation*}
    \widetilde{\eta}_A^j(f) = \sum_{k=1}^s
    \widetilde{\eta}_A^k(g_j^2(f-\psi_A^j(f))) =
    \eta_B(g_j^2(f-\psi_A^j(f))) = \eta_A^j(f).
  \end{equation*}
  This shows that $\widetilde{\eta}_A^j = \eta_A^j$ for
  $j=1,\ldots,s$.
\end{proof}

\begin{remark*}
  If we denote by $\operatorname{Der}_\psi(A)$ the linear space of
  derivations of $A$ at $\psi\in\M(A)$, then the lemma above shows
  that
  \begin{equation*}
    \operatorname{Der}_{\psi_B}(B) \cong \bigoplus_{\psi \in
      (i^*)^{-1}(\{\psi_B\})} \operatorname{Der}_\psi(A).
  \end{equation*}
\end{remark*}

\section{The proofs of Theorems \ref{analytic-curve} and
  \ref{extension}}
\label{more-proofs}

This section is devoted to proving Theorems \ref{analytic-curve} and
\ref{extension}.  We use results by Gamelin \cite{Gamelin} on finite
codimensional subalgebras of uniform algebras.  In particular, we need
to use his concept of a \emph{$\theta$-subalgebra}, which we define
here in the context of unital commutative Banach algebras.  If $A$ is
a Banach algebra and $\theta \in \M(A)$, a $\theta$-subalgebra of $A$
is a subalgebra $H$ such that there is a chain $H = H_0 \subset H_1
\subset \cdots \subset H_n = A$ where $H_k$ is the kernel of some
derivation in $H_{k+1}$ at the point $\theta$ (the restriction map
$i^*:\M(H_{k+1}) \to \M(H_{k})$ is a bijection, so the maximal ideal
spaces of all the algebras $H_k$ can be viewed as being the same).

The main result of~\cite{Gamelin} concerning finite codimensional
subalgebras is that, roughly speaking, every such subalgebra can be
constructed by first passing to a glued subalgebra and then taking the
intersection of a finite number of $\theta_j$-subalgebras of the glued
subalgebra.

\begin{lemma}
  \label{lemma-vanishing}
  Assume that $\Phi : \overline{\Omega} \to \overline{\D}^n$ is
  admissible.  There is a finite set $X \subset \Omega$ and
  $N\in\mathbb{N}$ such that if $f \in H^\infty(\Omega)$ has a zero of
  order $N$ at each point of $X$, then $f \in \mathcal{H}_\Phi$.  The
  same is true if one replaces $H^\infty(\Omega)$ by
  $A(\overline{\Omega})$ and $\mathcal{H}_\Phi$ by $\mathcal{A}_\Phi$.
\end{lemma}

\begin{proof}
  According to~\cite{Gamelin}*{Theorem 9.8}, there is a glued
  subalgebra $H_0$ of $H^\infty(\Omega)$ and a finite family of
  $\theta_j$-subalgebras $H_j$ of $H_0$ such that
  \begin{equation*}\mathcal{H}_\Phi =
    H_1\cap \cdots \cap H_r.\end{equation*}
  (Here $\theta_j \in \M(H_0)$.)

  Put $G = G(H^\infty(\Omega),H_0)$.  By Lemma~\ref{lemma-glued},
  $\psi(\z) \in \Omega$ for every $\psi \in G$.  Consider the map
  $i^*:\M(H^\infty(\Omega)) \to \M(H_0)$ and put $Y =
  (i^*)^{-1}(\{\theta_1,\ldots,\theta_r\})$.  By
  Lemma~\ref{glued-preimage}, $Y$ is a finite set.

  If $\psi \in Y$, then $\psi(\z)\in\Omega$, since either $\psi \in G$
  or $\psi$ is the unique preimage of some $\theta_j$.  In the latter
  case, since $\mathcal{H}_\Phi \subset H_j$, there is some derivation
  $\eta$ of $H_0$ at $\theta_j$ such that $\eta|\mathcal{H}_\Phi = 0$
  but $\eta\neq 0$.  By Lemma~\ref{glued-derivations}, $\eta$ extends
  to a derivation $\widehat{\eta}$ of $H^\infty(\Omega)$ at $\psi$.
  By Lemma~\ref{lemma-zero}, $\psi(\z) \in \Omega$, because
  $\widehat{\eta}|\mathcal{H}_\Phi = 0$.

  We claim that $X = \{\psi(\z) : \psi \in G\cup Y\} \subset \Omega$
  is the desired set.  Note that $X$ is finite.  Also, if $f \in
  H^\infty(\Omega)$ vanishes on $X$ then $f \in H_0$, because $\psi(f)
  = 0$ for every $\psi\in G$.

  Let $k_j$ be the codimension of $H_j$ in $H_0$.  Assume that $f \in
  H^\infty(\Omega)$ has a zero of order $2^{k_j}$ at every point of
  $X$.  Then $f$ can be factored as $f = f_1\cdots f_{2^{k_j}}$, where
  each of the $2^{k_j}$ functions belongs to $H^\infty(\Omega)$ and
  vanishes on $X$.  This implies $f \in H_j$ by~\cite{Gamelin}*{Lemma
    9.3}.  Thus, the lemma holds with $N = \max_j 2^{k_j}$.

  The proof for $A(\overline{\Omega})$ is similar.
\end{proof}

\begin{proof}[of Theorem~\ref{analytic-curve}]
  We only show that $\Phi^* H^\infty(\V) = \mathcal{H}_\Phi$ as the
  proof of $\Phi^*A(\overline{\Omega}) = \mathcal{A}_\Phi$ is
  identical.  The inclusion $\mathcal{H}_\Phi \subset
  \Phi^*H^\infty(\V)$ follows from~\eqref{eq:inclusions}.  We examine
  the reverse inclusion.

  Let $n \in \N$ be the integer and $X = \{z_1,\ldots,z_r\} \subset
  \Omega$ the finite set from Lemma~\ref{lemma-vanishing}.  Put $w_j =
  \Phi(z_j) \in \V$, $j=1,\ldots,r$.  Take $F \in H^\infty(\V)$ and
  observe that $F$ extends to an analytic function, which we also
  denote by $F$, on a neighborhood in $\C^n$ of each of the points
  $w_j$, $j=1,\ldots,r$.  Let $P \in \C[z_1,\ldots,z_n]$ be a
  polynomial such that
  \begin{equation*}
    D^\alpha P(z_j) = D^\alpha F(z_j), \qquad 0\leq|\alpha|\leq N,\
    j=1,\ldots,r.
  \end{equation*}

  Then $\Phi^*(F-P) = F\circ\Phi - P\circ\Phi$ has a zero of order $N$
  at each point of $X$, so $\Phi^*(F-P) \in \mathcal{H}_\Phi$.  Also,
  $\Phi^* P \in \mathcal{H}_\Phi$, because it is a polynomial in
  $\varphi_1,\ldots,\varphi_n$.  It follows that $\Phi^* F \in
  \mathcal{H}_\Phi$.
\end{proof}

\begin{proof}[of Theorem~\ref{extension}]
  By Lemma~\ref{finite-codim}, $\mathcal{H}_\Phi$ is a closed subspace
  of $H^\infty(\Omega)$.  Define an operator $L : \mathcal{H}_\Phi \to
  \mathcal{H}_\Phi$ by $Lf = \sum_{k=1}^nF_k(f) \circ \varphi_k$,
  where $F_1,\ldots,F_n$ are the operators that appear in
  Theorem~\ref{compact-operator}.  Since $I - L$ is compact, by the
  Fredholm theory there are bounded operators $R,P : \mathcal{H}_\Phi
  \to \mathcal{H}_\Phi$ such that $P$ has finite rank and $I = LR +
  P$.  The operator $P$ can be written as
  \begin{equation*}
    Pf = \sum_{k=1}^r \alpha_k(f)g_k,\qquad f\in H^\infty(\Omega),
  \end{equation*}
  for some $g_k \in \mathcal{H}_\Phi$ and $\alpha_k \in
  \mathcal{H}_\Phi^*$.  The functions $g_k$ can be expressed as
  \begin{equation*}
    g_k = \sum_{j=1}^l \prod_{i=1}^n f_{j,i,k}\circ\varphi_i
  \end{equation*}
  (so as to have the same number of multiplicands and terms in these
  sums, we can take some of the $f_{j,i,k}$ equal to $0$ or $1$).

  Take an $f \in H^\infty(\V)$.  By Theorem~\ref{analytic-curve},
  $\Phi^* f \in \mathcal{H}_\Phi$.  Put
  \begin{equation*}
    F(z_1,\ldots,z_n) = \sum_{k=1}^n F_k(R \Phi^*f)(z_k) +
    \sum_{k=1}^r\sum_{j=1}^l \alpha_k(\Phi^* f)\prod_{i=1}^n
    f_{j,i,k}(z_i).
  \end{equation*}
  Then $\Phi^* F = F \circ \Phi = LR\Phi^*f + P\Phi^*f = \Phi^*f$, so
  $F|\V = f$.

  If $g \in H^\infty(\D)$, then $\|g(z_k)\|_{\SA(\D^n)} =
  \|g\|_\infty$, so
  \begin{equation*}
    \begin{split}
      \|F\|_{\SA(\D^n)} &\leq \sum_{k=1}^n \|F_k(R\Phi^*f)\|_\infty +
      \sum_{k=1}^r\sum_{j=1}^l |\alpha_k(\Phi^*f)|\cdot\prod_{i=1}^n
      \|f_{j,i,k}\|_\infty\\
      &\leq C\|\Phi^*f\|_{H^\infty(\D)} = C\|f\|_{H^\infty(\V)}.
    \end{split}
  \end{equation*}
\end{proof}

\section{Continuous families of admissible functions}
\label{continuous-families}

In this section we prove a lemma concerning a family of admissible
functions $\{\Phi_\varepsilon\}$ which depends continuously on
$\varepsilon$.  The lemma shows that one can get operators
$F_k^\varepsilon$ as in Theorem~\ref{compact-operator} that satisfy
certain continuity properties in $\varepsilon$.  This result is used
in an application to the study of complete $K$-spectral sets in our
forthcoming article~\cite{Article2}.

\begin{lemma}
  \label{lema-13}
  Let $\Phi_\varepsilon= \big(\varphi_1^\varepsilon, \dots,
  \varphi_n^\varepsilon\big): \overline{\Omega}\to\overline{\D}^n$,
  $0\leq\varepsilon\leq\varepsilon_0$ be a collection of functions.
  Assume that $\Psi_\varepsilon$ is admissible for every
  $\varepsilon$, and, moreover, that one can choose sets $\Omega_k$ in
  the definition of an admissible collection so as not to depend on
  $\varepsilon$.  Assume that $\varphi_k^\varepsilon \in
  \Cont^{1+\alpha}(\Omega_k)$, with $0 < \alpha < 1$, and that the
  mapping $\varepsilon \mapsto \varphi_k^\varepsilon$ is continuous
  from $[0,\varepsilon_0]$ to $\Cont^{1+\alpha}(\Omega_k)$.

  Then there exist bounded linear operators $F_k^\varepsilon :
  A(\overline{\Omega}) \to A(\overline{\D})$, such that for
  \begin{equation*}
    L_\varepsilon(f) = \sum
    F_k^\varepsilon(f)\circ\varphi_k^\varepsilon,
  \end{equation*}
  and $0\leq\varepsilon\leq \varepsilon_0$, $L_\varepsilon - I$ is a
  compact operator on $A(\overline{\Omega})$, the mapping $\varepsilon
  \mapsto L_\varepsilon$ is norm continuous, and $\|F_k^\varepsilon\|
  \leq C$ for $k = 1,\ldots,n$, where $C$ is a constant independent of
  $k$ and $\varepsilon$.
\end{lemma}

The proof of Lemma~\ref{lema-13} uses the following technical fact:

\begin{lemma}
  \label{lemma-epsilon}
  Let $\Omega$ be a bounded domain satisfying the inner chord-arc
  condition,
  $\{\varphi_\varepsilon\}_{0\leq\varepsilon\leq\varepsilon_0} \subset
  A(\overline{\Omega})$ with $\varphi_\varepsilon'$ of class H\"older
  $\alpha$ in $\Omega$ and such that the mapping $\varepsilon \mapsto
  \varphi_\varepsilon$ is continuous from $[0,\varepsilon_0]$ to
  $\Cont^{1+\alpha}(\Omega)$.  Let $K \subset \overline{\Omega}$ be
  compact.  Assume that $\varphi_\varepsilon(\zeta) \neq
  \varphi_\varepsilon(z)$ if $\zeta\in K$, $z \in
  \overline{\Omega}\setminus\{\zeta\}$ and
  $0\leq\varepsilon\leq\varepsilon_0$.  Assume also that for each
  $0\leq\varepsilon\leq 1$, $\varphi_\varepsilon'$ does not vanish
  in~$K$.  Then for $\zeta \in K$, $z \in
  \overline{\Omega}\setminus\{\zeta\}$, and
  $\varepsilon,\delta\in[0,\varepsilon_0]$,
  \begin{equation*}
    \left|
      \frac{\varphi_{\varepsilon}'(\zeta)}{\varphi_{\varepsilon}(\zeta)
        - \varphi_{\varepsilon}(z)} -
      \frac{\varphi_{\delta}'(\zeta)}{\varphi_{\delta}(\zeta) -
        \varphi_{\delta}(z)} \right| \leq C\|\varphi_{\varepsilon}' -
    \varphi_{\delta}'\|_{\Cont^\alpha}|\zeta - z|^{\alpha-1} .
  \end{equation*}
\end{lemma}

\begin{proof}
  First check that
  \begin{equation}
    \label{eq:lemma-epsilon-*}
    \left|\frac{\varphi_\varepsilon(\zeta) -
        \varphi_\varepsilon(z)}{\zeta - z}\right| \geq C > 0,
    \qquad
    \zeta \in K,\ z \in \overline{\Omega}\setminus\{\zeta\},\
    0\leq\varepsilon\leq\varepsilon_0.
  \end{equation}
  To see this, put
  \begin{equation*}
    h(\zeta,z,\varepsilon) =
    \begin{cases}
      \frac{\varphi_\varepsilon(\zeta) - \varphi_\varepsilon(z)}{\zeta
        - z}, & \text{if } \zeta \in K, z \in
      \overline{\Omega}\setminus\{\zeta\},\\
      \varphi_\varepsilon'(\zeta), & \text{if } \zeta = z \in K,
    \end{cases}
  \end{equation*}
  which is continuous on the compact set $K \times \overline{\Omega}
  \times [0,\varepsilon_0]$.  As $h$ does not vanish, $|h| \geq C >
  0$, implying~\eqref{eq:lemma-epsilon-*}.

  Since $|\psi(u) - \psi(\zeta)| \leq
  \|\psi\|_{\Cont^\alpha}|u-\zeta|^\alpha$ and $|\psi(u)| \leq
  \|\psi\|_{\Cont^\alpha}$,
  \begin{equation*}
    \begin{split}
      |\varphi_{\varepsilon}'&(\zeta)\varphi_{\delta}'(u) -
      \varphi_{\delta}'(\zeta)\varphi_{\varepsilon}'(u)|\\
      &= \Big|[\varphi_{\varepsilon}'(u) -
      \varphi_{\varepsilon}'(\zeta)] [\varphi_{\delta}'(\zeta) -
      \varphi_{\varepsilon}'(\zeta)] + \varphi_{\varepsilon}'(\zeta)
      [\varphi_{\varepsilon}'(u) - \varphi_{\varepsilon}'(\zeta)
      + \varphi_{\delta}'(\zeta) - \varphi_{\delta}'(u)]\Big|\\
      &\leq \Big|\varphi_{\varepsilon}'(u) -
      \varphi_{\varepsilon}'(\zeta)\Big| \Big|\varphi_{\delta}'(\zeta)
      - \varphi_{\varepsilon}'(\zeta)\Big| +
      \big|\varphi_{\varepsilon}'(\zeta)\big|\, \Big|
      (\varphi_\varepsilon' - \varphi_\delta')(u) -
      (\varphi_\varepsilon' - \varphi_\delta')(\zeta)
      \Big|\\
      &\leq \|\varphi_{\varepsilon}'\|_{\Cont^\alpha}|u-\zeta|^\alpha
      \|\varphi_{\varepsilon}'-\varphi_{\delta}'\|_{\Cont^\alpha} +
      \|\varphi_{\varepsilon}'\|_{\Cont^\alpha}
      \|\varphi_{\varepsilon}'-\varphi_{\delta}'\|_{\Cont^\alpha}
      |u-\zeta|^\alpha.
    \end{split}
  \end{equation*}
  But
  \begin{equation*}
    \varphi_{\varepsilon}'(\zeta)[\varphi_{\delta}(\zeta)-
    \varphi_{\delta}(z)]
    -
    \varphi_{\delta}'(\zeta)[\varphi_{\varepsilon}(\zeta)-
    \varphi_{\varepsilon}(z)]
    =
    \int_{\gamma(z,\zeta)}\left(
      \varphi_{\varepsilon}'(\zeta)\varphi_{\delta}'(u) -
      \varphi_{\delta}'(\zeta)\varphi_{\varepsilon}'(u)\right) \,du,
  \end{equation*}
  and so
  \begin{equation*}
    \begin{split}
      \Big| \varphi_{\varepsilon}'(\zeta)[\varphi_{\delta}(\zeta)-
      \varphi_{\delta}(z)] -
      \varphi_{\delta}'(\zeta)[\varphi_{\varepsilon}(\zeta)-
      \varphi_{\varepsilon}(z)] \Big| &\leq
      C\|\varphi_{\varepsilon}'-\varphi_{\delta}'\|_{\Cont^\alpha}
      \int_{\gamma(z,\zeta)} |u-\zeta|^\alpha\,|du| \\
      &\leq C\|\varphi_{\varepsilon}'-
      \varphi_{\delta}'\|_{\Cont^\alpha} |z-\zeta|^{\alpha+1}.
    \end{split}
  \end{equation*}

  Combining this with~\eqref{eq:lemma-epsilon-*},
  \begin{equation*}
    \begin{split}
      \left| \frac{\varphi_{\varepsilon}'(\zeta)}
        {\varphi_{\varepsilon}(\zeta) - \varphi_{\varepsilon}(z)} -
        \frac{\varphi_{\delta}'(\zeta)}{\varphi_{\delta}(\zeta) -
          \varphi_{\delta}(z)} \right| &= \left|
        \frac{\varphi_{\varepsilon}'(\zeta)[\varphi_{\delta}
          (\zeta)-\varphi_{\delta}(z)] -
          \varphi_{\delta}'(\zeta)[\varphi_{\varepsilon}
          (\zeta)-\varphi_{\varepsilon}(z)]} {
          [\varphi_{\varepsilon}(\zeta)-\varphi_{\varepsilon}(z)]
          [\varphi_{\delta}(\zeta)-\varphi_{\delta}(z)] } \right|
      \\
      &\leq C\|\varphi_{\varepsilon}'-
      \varphi_{\delta}'\|_{\Cont^\alpha} \frac{|z-\zeta|^{\alpha+1}}{
        |\varphi_{\varepsilon}(\zeta)-\varphi_{\varepsilon}(z)|
        |\varphi_{\delta}(\zeta)-\varphi_{\delta}(z)| }
      \\
      &\leq C\|\varphi_{\varepsilon}'-
      \varphi_{\delta}'\|_{\Cont^\alpha} |z - \zeta|^{\alpha-1}.
    \end{split}
  \end{equation*}
\end{proof}

\begin{proof}[of Lemma~\ref{lema-13}]
  The construction of the functions $\eta_k$ and $\nu_k$ used in the
  proof of Theorem~\ref{compact-operator} depends solely on the
  geometry of $\Omega$, and not on the functions $\varphi_k$.  So
  define $F_k^\varepsilon$ by equation~\eqref{eq:thm2-**}, replacing
  $\varphi_k$ with $\varphi_k^\varepsilon$.  Then $L_\varepsilon - I$
  is compact by the proof of Theorem~\ref{compact-operator}.  We also
  define $A_k^\varepsilon$ and $B_k^\varepsilon$ by
  equation~\eqref{eq:AkBk}, with $\varphi_k^\varepsilon$ instead of
  $\varphi_k$.  Then by~\eqref{eq:thm2-**2},
  \begin{equation*}
    L_{\varepsilon}(f) - L_{\delta}(f) = \sum_{k=1}^n
    (A_k^{\delta}-A_k^{\varepsilon})(f) +
    (B_k^{\delta}-B_k^{\varepsilon})(\nu_kf).
  \end{equation*}

  Note that
  \begin{equation*}
    (A_k^{\delta}-A_k^{\varepsilon})(f)(z) = \CauchyMod f
    {\varphi_k^{\varepsilon}} {J_k}(z) - \CauchyMod f
    {\varphi_k^{\delta}} {J_k}(z) = \int_{J_k} \Big(
    \frac{(\varphi_k^\varepsilon)'(\zeta)}{\varphi_k^{\varepsilon}(\zeta)
      - \varphi_k^{\varepsilon}(z)} -
    \frac{(\varphi_k^\delta)'(\zeta)}{\varphi_k^{\delta}(\zeta) -
      \varphi_k^{\delta}(z)}\Big) f(\zeta)\,d\zeta.
  \end{equation*}
  Using Lemma~\ref{lemma-epsilon} and the fact that
  \begin{equation*}
    \int_{J_k} |z-\zeta|^{\alpha-1}\,|d\zeta| < \infty,
  \end{equation*}
  we have $\|A_k^{\delta} - A_k^{\varepsilon}\| \leq
  C\|(\varphi_k^\varepsilon)' - (\varphi_k^\delta)'\|_{\Cont^\alpha}$.
  Also, $\|B_k^{\delta} - B_k^{\varepsilon}\| \leq
  C\|(\varphi_k^\varepsilon)' - (\varphi_k^\delta)'\|_{\Cont^\alpha}$
  by similar reasoning.  These inequalities imply that $L_\varepsilon$
  depends continuously on $\varepsilon$ in the norm topology.

  To see that $\|F_k^\varepsilon\|\leq C$, with $C$ independent of
  $\varepsilon$, one must examine the proofs of
  Theorem~\ref{compact-operator} and Lemma~\ref{lemma-simple} to check
  that the constants that bound the operators which appear there can
  be taken to be independent of $\varepsilon$.  First, we give the
  details concerning the proof of Lemma~\ref{lemma-simple}.

  Instead of~\eqref{eq:wsing}, we require the inequality
  \begin{equation}
    \label{eq:lemma-13-*}
    \left|
      \frac{(\varphi_k^{\varepsilon})'(\zeta)}{\varphi_k^\varepsilon(
        \zeta)-\varphi_k(z)}
      -
      \frac{1}{\zeta-z}
    \right| \leq C|\zeta - z|^{\alpha-1},\qquad \zeta \in J_k,\
    z \in \overline{\Omega}_k\setminus\{\zeta\},\ 0\leq\varepsilon\leq
    \varepsilon_0.
  \end{equation}
  Here $C$ should be a constant independent of $\varepsilon$, so we
  cannot simply apply Lemma~\ref{weak-singularity}.  To prove this
  inequality, apply Lemma~\ref{weak-singularity} to $\varphi_k^0$ and
  get~\eqref{eq:lemma-13-*} for $\varepsilon = 0$, and then use
  Lemma~\ref{lemma-epsilon} to obtain
  \begin{equation*}
    \left|
      \frac{(\varphi_k^{\varepsilon})'(\zeta)}{\varphi_k^{\varepsilon}(
        \zeta) - \varphi_k^{\varepsilon}(z)} -
      \frac{(\varphi_k^0)'(\zeta)}{\varphi_k^0(\zeta) - \varphi_k^0(z)}
    \right| \leq C|\zeta - z|^{\alpha-1}, \quad \zeta \in J_k,\ z \in
    \overline{\Omega}_k\setminus\{\zeta\},\
    0\leq\varepsilon\leq\varepsilon_0,
  \end{equation*}
  where $C$ is independent of $\varepsilon$.
  Then~\eqref{eq:lemma-13-*} follows from the triangle inequality.

  Let $\widehat{J}_k$ and $\psi$ be as in the statement of
  Lemma~\ref{lemma-simple}.  We verify that $\|\CauchyMod {\psi}
  {\varphi_k^\varepsilon} {\widehat{J}_k}\|_\infty \leq C$, with $C$
  independent of $\varepsilon$.  By the proof of
  Lemma~\ref{lemma-simple}, $\varphi_k^\varepsilon(\Omega_k)$ is an
  open set containing $\varphi_k^\varepsilon(\widehat{J}_k)$.
  Moreover, it follows from the continuity of $\varphi_k^\varepsilon$
  in $\varepsilon$ and the compactness of the interval
  $[0,\varepsilon_0]$ that the distance from
  $\varphi_k^\varepsilon(\widehat{J}_k)$ to the boundary of
  $\varphi_k^\varepsilon(\Omega_k)$ is bounded below by a positive
  constant independent of $\varepsilon$.  Therefore, it suffices to
  show that $\CauchyMod {\psi} {\varphi_k^\varepsilon}
  {\widehat{J}_k}$ is bounded in $\varphi_k^\varepsilon(\Omega_k)
  \setminus \varphi_k^\varepsilon(\widehat{J}_k)$ by a constant
  independent of $\varepsilon$, because when the distance from some
  point $z$ to $\varphi_k^\varepsilon(\widehat{J}_k)$ is greater than
  a constant, $\CauchyMod {\psi} {\varphi_k^\varepsilon}
  {\widehat{J}_k} (z)$ is readily bounded by a constant independent of
  $z$ and $\varepsilon$.

  To show that $\CauchyMod {\psi} {\varphi_k^\varepsilon}
  {\widehat{J}_k}$ is bounded in $\varphi_k^\varepsilon(\Omega_k)
  \setminus \varphi_k^\varepsilon(\widehat{J}_k)$, we prove as in
  Lemma~\ref{lemma-simple} that $\CauchyMod {\psi}
  {\varphi_k^\varepsilon}{\widehat{J}_k}\circ\varphi_k^\varepsilon$ is
  bounded in $\Omega_k\setminus\widehat{J}_k$.
  Write~\eqref{eq:lemma-simple-*} for $\varphi_k^\varepsilon$ instead
  of $\varphi_k$ and $\widehat{J}_k$ instead of $\Gamma$, and then
  use~\eqref{eq:lemma-13-*} to obtain
  \begin{equation*}
    \left|\CauchyMod
      {\psi}{\varphi_k^\varepsilon}{\widehat{J}_k}(\varphi_k(z)) -
      \Cauchy {\psi}{\widehat{J}_k}(z) \right| \leq C\|\psi\|_\infty
    \int_{\widehat{J}_k} |\zeta - z|^{\alpha-1}\,d\zeta \leq
    C\|\psi\|_\infty,
  \end{equation*}
  where $C$ is independent of $\varepsilon$.  Since $\Cauchy
  {\psi}{\widehat{J}_k} \in H^\infty(\C\setminus\widehat{J}_k)$, we
  get the required bound.

  It remains to check that the $H^\infty$ norms in Claim~\ref{claim1}
  (see the proof of Theorem~\ref{compact-operator}) can be bounded by
  a constant independent of $\varepsilon$.  We can apply methods
  similar to the ones that we have used for Lemma~\ref{lemma-simple}.
  Define $(G_k^-)^\varepsilon$ as in~\eqref{eq:gk}, replacing
  $\varphi_k$ with $\varphi_k^\varepsilon$.  Put $g_k^\varepsilon =
  (G_k^-)^\varepsilon(f)$.  This is in
  $H^\infty(\C\setminus\varphi_{k+1}^\varepsilon(J_{k+1}^-\cup
  J_k^R))$ by Claim~\ref{claim1}.  We want to show that
  $g_k^\varepsilon$ is bounded by a constant independent of
  $\varepsilon$.

  Define $h_k^+$ and $h_k^-$ as in the proof of Claim~\ref{claim1}
  (these functions do not depend on $\varphi_k$).  Compute
  \begin{equation*}
    \begin{split}
      g_k^\varepsilon \circ \varphi_{k+1}^\varepsilon &- (h_k^- +
      h_k^+ \circ
      R_k^{-1}) =\\
      & [\CauchyMod {\nu_k f} {\varphi_{k+1}^\varepsilon} {J_{k+1}^-}
      \circ \varphi_{k+1}^\varepsilon - \Cauchy {\nu_k f} {J_{k+1}^-}]
      + [ \CauchyMod {(\nu_kf)\circ R_k^{-1}}
      {\varphi_{k+1}^\varepsilon} {J_k^R} \circ
      \varphi_{k+1}^\varepsilon - \Cauchy {(\nu_kf)\circ R_k^{-1}}
      {J_k^R} ].
    \end{split}
  \end{equation*}
  Arguing as before and using~\eqref{eq:lemma-13-*}, each of the two
  terms in brackets is bounded by a constant independent of $\varphi$.
  Since $h_k^- + h_k^+\circ R_k^{-1} \in
  H^\infty(\C\setminus(J_k^-\cup J_k^R))$, $g_k^\varepsilon \circ
  \varphi_{k+1}^\varepsilon$ is uniformly bounded in
  $S_{k+1}^-\setminus(J_{k+1}^- \cup J_k^R)$, and so $g_k^\varepsilon$
  is uniformly bounded in
  $\varphi_{k+1}^\varepsilon(S_{k+1}^-\setminus(J_{k+1}^-\cup
  J_k^R))$.

  Now choose open circular sectors $S_\varepsilon$ and
  $S_\varepsilon'$ with vertex on $w_k^\varepsilon =
  \varphi_{k+1}^\varepsilon(z_k)$ such that
  $\overline{S_\varepsilon}\cap \overline{S_\varepsilon'} =
  \{w_k^\varepsilon\}$, and satisfying the following conditions (see
  Figure~\ref{fig:2} in the proof of Theorem~\ref{compact-operator}):
  \begin{itemize}
  \item $\varphi_{k+1}^\varepsilon(J_{k+1}^- \cup J_k^R) \subset
    S_\varepsilon' \cup \{w_k^\varepsilon\}$.
  \item $\D_{\varepsilon_0}(w_k^\varepsilon)\setminus
    \varphi_{k+1}^\varepsilon(S_{k+1}^-) \subset S_\varepsilon \cup
    \{w_k^\varepsilon\}$, for some $\varepsilon_0 > 0$.
  \item The straight edges of $S_\varepsilon$ are contained in
    $\varphi_{k+1}^\varepsilon(S_{k+1}^-)\cup \{w_k^\varepsilon\}$.
  \end{itemize}
  This can be done by the continuity of $\varphi_{k+1}^\varepsilon$ in
  $\varepsilon$ and by shrinking $V_k$ if necessary.

  To show that $g_k^\varepsilon$ is bounded in $S_\varepsilon$, we use
  the Phragm\'en-Lindel\"of principle as in the proof of
  Claim~\ref{claim1}.  There, we proved that $g_k^\varepsilon$ is
  $O(|z-w_k^\varepsilon|^{-1})$ as $z\to w_k^\varepsilon$.  Thus,
  $g_k^\varepsilon$ is bounded in $S_\varepsilon$ by the supremum of
  $|g_k^\varepsilon|$ on the straight edges of $S_\varepsilon$.  Since
  these straight edges are contained in
  $\varphi_{k+1}^\varepsilon(S_{k+1})$ and we had a bound for
  $\varphi_k^\varepsilon$ which is uniform in $\varepsilon$ on this
  set, there is a bound on $S_\varepsilon$ which is also uniform in
  $\varepsilon$.

  Clearly, $g_k^\varepsilon(z)$ is uniformly bounded in $\varepsilon$
  and $z$ when the distance from $z$ to
  $\varphi_{k+1}^\varepsilon(J_{k+1}^- \cup J_k^R)$ is greater than a
  positive constant.  Also, $g_k^\varepsilon$ is uniformly bounded on
  $U_\varepsilon \setminus \varphi_{k+1}^\varepsilon(J_{k+1}^- \cup
  J_k^R)$, where $U_\varepsilon$ is some open set containing
  $\varphi_{k+1}^\varepsilon(J_{k+1}^- \cup J_k^R)$ and such that the
  distance from $\partial U_\varepsilon$ to
  $\varphi_{k+1}^\varepsilon(J_{k+1}^- \cup J_k^R)$ is bounded below
  by a positive constant independent of $\varepsilon$.  This finishes
  the proof.
\end{proof}

\affiliationone{
  Michael A. Dritschel\\
  School of Mathematics {\&} Statistics\\
  Newcastle University\\
  Newcastle upon Tyne\\
  NE1 7RU\\
  UK\\
  \email{michael.dritschel@newcastle.ac.uk}} \affiliationtwo{
  Daniel Est\'evez\\
  Departamento de Matem\'{a}ticas\\
  Universidad Aut\'onoma de Madrid\\ Cantoblanco 28049 (Madrid)\\
  Spain\\ \email{daniel.estevez@uam.es}} \affiliationthree{
  Dmitry Yakubovich\\
  Departamento de Matem\'{a}ticas\\
  Universidad Aut\'onoma de Madrid\\ Cantoblanco 28049 (Madrid)\\ Spain\\
  and Instituto de Ciencias Matem\'{a}ticas (CSIC - UAM - UC3M - UCM)\\
  \email{dmitry.yakubovich@uam.es}}

\end{document}